\documentclass{amsart}
\usepackage{setspace}

\usepackage{frenchineq}
\usepackage{amsmath}
\usepackage{amsthm}
\usepackage{amssymb}
\usepackage{verbatim}

\usepackage{subfigure}
\usepackage{verbatim}
\usepackage{url}
\numberwithin{equation}{section}
\usepackage{textcomp}

\usepackage{xy}
\usepackage{tikz}
\usepackage{color}

\usepackage{epstopdf}
\usepackage{mathrsfs}

\usepackage{amsmath}
\usepackage{amsthm}
\usepackage{amssymb}
\usepackage{verbatim}
\usepackage{subfigure}
\usepackage{verbatim}
\usepackage{url}
\numberwithin{equation}{section}
\usepackage{textcomp}
\usepackage{stmaryrd}
\usepackage{xy}
\usepackage{color}

\usepackage{epstopdf}
\usepackage{mathrsfs}
\usepackage{hyperref}
\usepackage{cleveref}
\usepackage{tikz}
\usepackage{color}

\usepackage{bbm}    

\usepackage{bm}

\usepackage[colorinlistoftodos,bordercolor=orange,backgroundcolor=orange!20,linecolor=orange,textsize=scriptsize]{todonotes}

\definecolor{Mygrey}{gray}{0.75}

\def\displayandname#1{\rlap{$\displaystyle\csname #1\endcsname$}%
                      \qquad \texttt{\char92 #1}}

\def\circi{{\bigcirc\!\!\!\! i\,}}
\def\circj{{\bigcirc\!\!\!\! j\,}}

\newtheorem*{lemma*}{Lemma}

\newcommand{\bi}{\mathbf{i}}
\newcommand{\bbi}{\bar{\mathbf{i}}}
\newcommand{\sbi}{\mathbf{i}^\star}
\newcommand{\bj}{\mathbf{j}}
\newcommand{\bk}{\mathbf{k}}
\newcommand{\bs}{\mathbf{s}}

\newcommand{\puiseuxA}{{\bm A}}

\newcommand{\puiseuxf}{{\bm f}}

\newcommand{\barred}[1]{{\,\bar #1\,}}
\makeatletter
\newcommand{\circled}[1]{{\,\bigcirc\hspace{-.123in} #1\,}}

\makeatother

\definecolor{Mygrey}{gray}{0.75}

\def\displayandname#1{\rlap{$\displaystyle\csname #1\endcsname$}%
                      \qquad \texttt{\char92 #1}}

\makeatletter
\def\url@leostyle{%
  \@ifundefined{selectfont}{\def\UrlFont{\sf}}{\def\UrlFont{\small\ttfamily}}}
\makeatother

\DeclareMathAlphabet{\mathbbold}{U}{bbold}{m}{n}
\newcommand{\zero}{\mathbbold{0}}
\newcommand{\unit}{\mathbbold{1}}

\newcommand{\per}{\operatorname{per}}
\newcommand{\Id}{\operatorname{Id}}

\newcommand{\rmax}{\mathbb{R}_{\max}}
\newcommand{\R}{\mathbb{R}}
\newcommand{\K}{\mathbb{K}}

\newcommand\val{{\operatorname{val}}}
\newcommand\TP{{\mathsf{TP}}}
\newcommand\TN{{\mathsf{TN}}}
\newcommand\UM{{\mathsf{UM}}}
\newcommand\trop{{\operatorname{trop}}}

\newcommand\GL{{\operatorname{GL}}}

\newtheorem{thm}{Theorem}[section]
\newtheorem{pro}[thm]{Proposition}

\newtheorem{cor}[thm]{Corollary}

\theoremstyle{definition}
\newtheorem{df}[thm]{Definition}

\theoremstyle{remark}
\newtheorem{rem}[thm]{Remark}
\newtheorem{exa}[thm]{Example}

\title{Tropical planar networks}

\author{{S}t\'ephane Gaubert}
\address{St\'ephane Gaubert,
INRIA Saclay--\^Ile-de-France and CMAP, \'Ecole 
polytechnique, IP Paris, CNRS. Address: CMAP, \'Ecole polytechnique,
Route de Saclay,
91128 Palaiseau Cedex, France.}
\email{Stephane.Gaubert@inria.fr}

\author{{A}di Niv}
\address{Adi Niv,
Mathematics Department, Science Faculty, Kibbutzim College.
Address: Kibbutzim College, 149 Namir Rd., Tel-Aviv, Israel.}
\email{adi.niv@smkb.ac.il}

\thanks{The first author acknowledges the support of the Gaspard Monge (PGMO) program of Fondation Math\'ematique Hadamard, EDF, Orange and Thales, of the ICODE institute of Paris-Saclay, also of Labex Hadamard. This work started when the second author was with INRIA, being supported by the Chateaubriand program from the French Ministery of Foreign Affairs and by an INRIA fellowship.}

\begin{document}

\begin{abstract}

  We show that every tropical totally positive matrix can be uniquely represented as the transfer matrix of a canonical totally connected weighted planar network.

  We deduce a uniqueness theorem for the factorization
  of a tropical totally positive in terms of elementary Jacobi matrices.

\vskip 0.15 truecm

\noindent \textit{Keywords:  Planar networks; total positivity; total nonnegativity;   compound matrix;  
permanent.}
\vskip 0.1 truecm

\noindent \textit{AMSC: 15A15 (Primary), 15A09, 15A18, 15A24, 15A29, 15A75, 15A80, 
15B99.} 	
\end{abstract}

\maketitle

\section{Introduction}
\subsection{Motivation and context}
A real matrix is said to be {\em totally positive} (resp.~{\em totally nonnegative}) 
if all its minors are positive (resp.~nonnegative).
Matrices that
are totally nonnegative and that have a totally positive power,
a.k.a, oscillatory matrices, go back to the work of Gantmacher and Krein, see~\cite{Gantmacher&Krein},
and also~\cite[\S4]{TPM}.
Totally nonnegative matrices arise in several classical fields,
like probability theory~\cite{KMc} or
approximation theory (see e.g.~\cite{kluwertotalpositivity}).
they have appeared
in the theory of canonical bases for quantum groups~\cite{berensteinparametrization}.
We refer the reader to the monograph by Fallat and Johnson~\cite{Fallat&Johnson} or to the survey by Fomin and Zelevinski~\cite{F&Z} for more information.

In~\cite{GN17}, we investigated the tropical analogues of totally
positive and totally nonnegative matrices. 
Tropical totally nonnegative matrices
can be defined as images by the nonarchimedean
valuation of totally nonnegative matrices over a real closed nonarchimedean field, like the field of Puiseux series with real coefficients. Alternatively,
tropical totally nonnegative matrices can be defined by requiring their minors
to be ``nonnegative'' in the tropical sense. We showed
that these two approaches yield the same
class of matrices, which is nothing but the set of
(opposite of) Monge matrices, a classical family of matrices
arising in combinatorial optimization and optimal transport~\cite{BKR}.
Monge matrices are defined by requiring only the tropical nonpositivity
of $2\times 2$ minors. They can be identified to submodular functions
defined on a finite two-dimensional grid.

A fundamental feature of classical totally nonnegative
matrices arises when considering
elementary operations on matrices, encoded
by Jacobi matrices:
any invertible totally nonnegative matrix
can be factored as a product of elementary Jacobi matrices
with nonnegative entries. This result is best understood
in terms of planar networks: one can associate to a weighted planar
network with $n$ sources and $n$ targets a {\em transfer matrix}.
This matrix encodes the sums of weights of paths between
pairs of source and target nodes. Then, any product
of Jacobi matrices can be realized as the transfer
matrix of a planar network. Furthermore, classical results 
of Loewner and Whitney show that there
is a canonical choice of planar network
leading to a parametrization result:
every totally positive matrix is represented
in a unique way by positive weights in this
planar network,
see~\cite{Loewner},~\cite{Whitney}, and Theorem~12 of~\cite{F&Z}.

We observed in~\cite{GN17} that, in the tropical setting,
totally nonnegative matrices can still be represented as tropical transfer
matrices of planar networks, or as a product of tropical elementary Jacobi matrices.
However, the uniqueness issue for such a representation
was not addressed in~\cite{GN17}.
Indeed, it is a rule in tropical algebra
that tropicalizing makes things more degenerate,
and uniqueness results do not always
carry over.

\subsection{Main results}
Our main result, \Cref{PN} solves the question of the uniqueness
of the weights in the representation of tropical totally
positive or nonnegative matrices by planar networks.
We show that a tropical totally positive matrix can be
represented in a unique way by a weight vector
of a canonical planar network.
Moreover, we characterize the weight vectors which arise in
this representation, by certain ``trapeze'' and ``parallelogram''
inequalities. This theorem reveals a fundamental discrepancy with the
classical theory, in which the weights are arbitrary positive numbers:
indeed, the classical parametrization of totally positive matrices
involves parameters that belong to the standard positive cone (the hyper-orthant). Here, we need to cut out the tropical analogue of this cone by
inequalities to obtain a bijective parametrizing set.
This operation eliminates regions of the parameter space representing matrices
which are tropical totally nonnegative but not totally positive.

As an application, we deduce from \Cref{PN} a uniqueness result for the
decomposition of a tropical totally positive matrix
as a product of elementary Jacobi matrices, see \Cref{cor-unique}.
We also deduce a result concerning the factorization
of totally positive matrices over a nonarchimedean valued field:
the valuation of the factors can be recovered
only from the valuation of the product, if the valuation
of this product is a tropical totally positive matrix.

All the previous results are stated for a canonical planar network
arising in the theory of totally positive matrices.
There are other planar networks,
corresponding to different ``factorization schemes'',
which lead to bijective parametrizations,
and
that different parametrizations
are related by birational subtraction-free transformations~\cite{F&Z}.
This entails that our uniqueness results hold as well for planar
networks arising all factorization schemes (\Cref{cor-uniquenew}).

\subsection{Related work}
Beside the fundamental results on classical total positivity,
on which we build (see e.g.~\cite{berensteinparametrization,F&Z,Fallat&Johnson}), our study is motivated or inspired, directly or indirectly,
by a series of works in tropical algebra and geometry.

To understand better the nature of \Cref{PN}, it
is useful to draw an analogy with Choquet's theory. 
  The latter studies the descriptions of points of convex
  set as the barycenter of measures supported by extreme points
  of the set; the tropical analogue of Choquet theory
  has received attention, both in finite dimension~\cite{GK,BSS} and infinite dimension~\cite{agw04b}.
  Whereas it is a trivial fact that a point in a classical simplex can be uniquely represented as a barycenter of its vertices, the tropical analogue
  is already less trivial: it is only true that a point in the {\em interior}
  of a tropical
  simplex can be uniquely represented as a tropical
  barycenter of the tropical vertices of this simplex.
  This follows for instance from the representation of tropical
  polyhedra as polyhedral complexes, see~\cite[Th.~15]{DS}.
  The representation of totally positive matrices by weights
  of planar networks can be thought of as nonlinear Choquet
  theorem, in which the representing measure is replaced
  by weights. Then, \Cref{PN} is remarkably analogous to the
  Choquet theorem for tropical simplices: the uniqueness
  of the representation only holds for a matrix in $\TP^\trop$,
  i.e., for a matrix in the {\em interior} of $\TN^\trop(\R)$.
  
  More generally, a motivation comes from the study of semialgebraic sets over nonarchimedean fields. The goal here is to understand what kind of properties of semialgebraic objects over a nonarchimedean valued field can be inferred
  using images by the nonarchimedean valuation.
  Results in this spirit go back to Develin and Yu~\cite{develin2007}, who
  showed that tropical polyhedra are images of nonarchimedean polyhedra~\cite{develin2007}. The tropicalization
  of the set of symmetric positive definite matrices has been characterized in~\cite{YU}. The tropicalization of general semialgebraic sets, including spectrahedra, has been studied 
  in~\cite{alessandrini2013,skomra,yuetal}. For instance, results
  of~\cite{skomra} show that under a genericity condition on the valuations
  of the input,  a set defined by a finite collection of polynomial inequalities  over a non-archimedean field is non-empty if and only if the set
  defining by ``tropicalizing'' these inequalities is non-empty.
  i.e., ``existence'' results tropicalize under appropriate genericity
  assumptions. It is natural to ask whether ``uniqueness' results of a semi-algebraic nature can be generally tropicalized under appropriate conditions, \Cref{PN} shows a special situation in which the answer
  is positive.

  Some different issues of uniqueness of representations
  have been studied,
  in the work on tropical Pl\"ucker
  functions by Danilov, Karzanov and Koshevoy~\cite{danilov}, i.e.,
  functions on some subset $B$ of $\mathbb{Z}^n$ that obey tropical analogs
  of the classical Pl\"ucker relations. This article determines
  bases, i.e. subsets $B'\subset B$ such that a tropical Pl\"ucker
  function is uniquely determined by its restriction to $B'$.

  We finally note that tropical total positivity has been studied in several works.
Most of them concern the tropicalization of the totally positive part of the Grassmannian, ~\cite{POST,SW05}.
In the tropical setting,  totally positive matrices are related to the totally positive Grassmanian, albeit in a weaker way than in the classical case.
Indeed, in the latter case, there is a canonical bijective transformation (Stiefel embedding) between the space of totally positive matrices of size $m\times n$ and the totally positive Grassmanian $\operatorname{Gr}^+_{m,m+n}$, described in~\cite{POST}. This transformation is only an injection in the tropical setting, see the discussion in~\cite{GN17}.

\section{Background materials}
In this section, we recall basic properties and constructions,
needed to state our results.
We first recall definitions and notation
concerning the tropical structure and the notion
of total positivity.

\subsection{Tropical algebra}

The max-plus 
(or tropical) semifield, denoted by~$\rmax$,
is the set~$\mathbb{R}\cup\{-\infty\}$ equipped with the laws~$a\oplus b:= \max(a,b)$ 
and~$a\odot b:= a+b$. 
(See for instance~\cite{BCOQ92,TAG,MPA,butkovicbook,MacStur}.) It has a zero element,~$\zero=-\infty$, and a unit 
element,~$\unit=0$. We abuse notation by using the same symbol, $\rmax$
for the semifield and for its ground set. 
 
The tropical numbers can be thought of as the images
by the valuation of the elements of a nonarchimedean field.
In this perspective, a convenient, concrete,
choice of ordered nonarchimedean field, denoted
by $\K$, consists of
(formal, generalized) Puiseux series with real coefficients and real exponents.
Such a series can be written as 
\begin{align}
\puiseuxf:= \sum_{k\geq 0} a_k t^{b_k}\enspace , 
\label{e-puiseux}
\end{align}
where~$a_k\in \R$, $b_k\in \R$, and~$(b_k)$ is a decreasing sequence
converging to~$-\infty$. 
The {valuation} of~$\puiseuxf$ is defined to be the largest exponent of~$\puiseuxf$, 
i.e., ~$\val (\puiseuxf):= \sup \{b_k \mid a_k\ne 0\}$,
with the convention that~$\val (0)=-\infty$.
A nonzero series is said to be {\em positive} if its leading coefficient  is positive. This field, or rather its complexification, $\K[\sqrt{-1}]$,
has been studied in~\cite{MARK}, as a canonical tool in tropical geometry.
Other natural choices of nonarchimedean fields are discussed
in~\cite{alessandrini2013,skomra,yuetal}, in particular, fields
of absolutely convergent Puiseux series are also covered
in these works. The fact that the value group of $\mathbb{K}$ is $\R$
(instead of $\mathbb{Q}$ for ordinary Puiseux series)
simplifies some statements. The fact that $\K$ is real closed
(and so has the same first order theory as the field of real numbers) is also
helpful.

In tropical algebra, we are interested in relations
between properties of objects defined over~$\mathbb{K}$ 
and their tropical analogues.
In particular, the {\em tropical permanent} of~$A$ is defined as 
\begin{align}
\per (A):= 
\max_{\sigma \in S_d} \sum_{i\in[d]} A_{i,\sigma(i)} \enspace ,
\label{e-def-tper}
\end{align}
where~$S_d$ is the set of permutations on~$[d]:=\{1,\dots,d\}$, and~$\sum_{i\in[d]} 
A_{i,\sigma(i)}$ is the \textit{weight} of the permutation~$\sigma$ in~$\per(A)$. 

We say that the $n\times n$ matrix~$A$ is (tropically) {\em sign-nonsingular}
if~$\per (A) \neq -\infty$ and if all the permutations~$\sigma$, such that $A_{1,\sigma(1)}+\cdots+A_{n,\sigma(n)}$ is of maximum weight, 
have the same parity. Otherwise,~$A$ is said to be (tropically) {\em sign-singular}. 
We refer the reader to~\cite{shader} for more background
on the classical notion of sign-nonsingularity, and to~\cite{NSM,benchimol2013} for its tropical version. When~$A$ is  sign-nonsingular, and $A_{ij} = \val \puiseuxA_{ij}$ for some $n\times n$ matrix $\puiseuxA$ with entries in $\K$,
it is easily seen that
\[
\val (\det(\puiseuxA)) = \per(A)  \enspace ,
\] 
and the sign of~$\det(\puiseuxA)$ coincide with the sign of every permutation of 
maximal weight in~$\per(A)$.
A \textit{tropical minor}   is defined as the tropical permanent of a square 
submatrix. A tropical minor 
is said to be  \textit{tropically positive}
(resp.~\textit{tropically negative}) if all its 
permutations of maximum weight   are even (resp.~odd). It is said to be 
\textit{tropically nonnegative} (resp.~\textit{tropically nonpositive}) 
if either the above condition
holds or the submatrix is  sign-singular. This terminology can be justified
by embedding the max-plus semiring in the symmetrized max-plus semiring~\cite{LS,AGG14}.

\subsection{Total positivity and total nonnegativity}

We denote by~$\TP$ (resp.~$\TN$) the set of totally positive (resp.~totally 
nonnegative) matrices over a field. These are matrices whose minors
are all positive (resp.~nonnegative).
The set~${\TP_t}$ 
(resp.~${\TN_t}$) denotes matrices whose minors of size at most~$t$ 
are   positive (resp.~nonnegative).
Similarly, we shall denote by~$\TP^{\trop}$ (resp.~$\TN^{\trop}$) the set of tropical totally positive (resp.~tropical totally nonnegative) matrices,
which have entries in~$\R$ (resp.~in~$\rmax$).  These are matrices,
such that, for every minor of the matrix, all the optimal
solutions of the corresponding optimal assignment
problem are {\em even} permutations (resp., there is
at least one optimal permutation which is even).
Observe that an entry of a matrix of $\TN^{\trop}$ belongs
to $\rmax=\R\cup\{-\infty\}$; i.e., we allow the value
$-\infty$ which is tropically nonnegative. In contrast,
an entry of a matrix of $\TP^{\trop}$ belongs
to $\R$ (all finite real numbers being tropically positive).
We denote  by~$\TN^\trop(\R)=\TN^\trop$  the subset of matrices in $\TN^\trop$ 
whose entries are finite.
We also denote by~${\TP^{\trop}_t}$ (resp.~${\TN^{\trop}_t}$) 
the set of matrices with entries in~$\R$ (resp.~in~$\rmax$) whose every 
tropical minor of size at most~$t$ is tropically 
positive (resp.~tropically nonnegative). We finally define
$\TN^{\trop}_t(\R)$ to be the subset of $\TN^\trop_t$ consisting
of matrices with finite entries.

\subsection{Main results of~\cite{GN17} required for this paper}
We studied in \cite{GN17} the images by the valuation of the 
classical classes of totally positive or totally nonnegative matrices
over~$\K$, 
and related the classical and tropical notions of total nonnegativity.  
We showed that the image by the valuation of the set of tropical totally positive
matrices is determined by the tropical nonnegativity of $2\times 2$ minors:
$$
\TN_2^\trop(\R)
 =\TN^\trop (\R)
= \val(\TN(\K^*)) =\val(\TP(\K)) \enspace,
$$
where $\TP(\K)$ denotes the set of totally positive matrices
with entries in the nonarchimedean field $\K$, and
$\TN(\K^*)$ denotes the set of totally nonnegative matrices
with entries in $\K^*$. 

We also showed that~$\TN^\trop_2(\R)$  is precisely the 
set of opposites of \textit{Monge matrices}, named after Gaspard Monge.
The set of Monge matrices has an explicit polyhedral parametrization
which follows from results of~\cite{BKR} and~\cite{FIEDLER}.

Another main result of~\cite{GN17} provides a tropical analogue
of a theorem of Loewner and Whitney~\cite{Whitney,Loewner}
and~\cite[Theorem~12]{F&Z}.
The classical theorem shows that any invertible totally nonnegative
matrix is a product of nonnegative elementary Jacobi matrices,
a similar property holds in the tropical setting:
$$
\val(\GL_n(\K)\cap\TN(\K))=
\langle\text{tropical~Jacobi~elementary~matrices}\rangle\enspace .
$$
More information on Jacobi matrices will be given in \S\ref{S3}.

\subsection{Planar networks}\label{stpn}
We next recall the correspondence between
totally nonnegative matrices and planar networks, referring
the reader to~\cite{F&Z,Fallat&Johnson} for more information.
\begin{df} Let~$G$ be a weighted directed graph, whose edges are equipped with real numbers called \textit{weights}.
We assume there are $n$ distinguished nodes called sources and $m$ other distinguished  nodes called targets. 
The \textit{weight of a  path} between two nodes
is the product of the weights of the edges of this path.
Similarly, the
\textit{tropical weight} of this path 
is the sum of the weights of the edges of this path.
The \textit{transfer matrix}  of~$G$ is the~$n\times m$ matrix
whose $i,j$ entry is the sum of weights of all paths from source node $i$ to target node $j$, for $i\in[n]$ and $i\in[m]$, with the convention
that this sum is~$0$  if such a path does not exist. 
In~\cite{F&Z}, the term weight matrix is used instead of transfer matrix.
We chose the latter term here, as it avoids the confusion with the matrix of weights $(w_{i,j})$ appearing below.
Similarly, the \textit{tropical transfer matrix} of~$G$ is the $n\times m$
matrix whose $i,j$ entry is the maximal tropical weight of a path
from source node $i$ to target node $j$;
by convention, this maximal weight is~$-\infty$ 
if such a path does not exist.
The \textit{length} of a path is the number of its edges.

A  graph is called \textit{planar} if it can be drawn on a plane so that its edges have only endpoint-intersections. A \textit{planar network} is a weighted directed 
planar graph, with no cycles. 
Throughout, we assume a network has~$n$ sources and $n$ targets, numbered bottom
to top, with edges
assigned with real weights. 
We call a planar network \textit{totally connected} if for any
set $I$ of source nodes, and every set $J$ of target node
such that~$|I|=|J|$ there exists a collection of vertex-disjoint paths
such that every node of~$I$ is connected to one node of~$J$ by one of these paths.
In all the planar examples which follow, the edges will be oriented from left to right, without indicating explicitly the orientation of the graph.
 \end{df}

\begin{exa}\label{ex-easy} 
  The following planar network
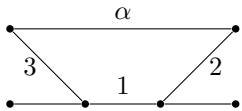
\begin{center}
  \begin{tikzpicture}[main_node/.style={circle,fill=black,minimum 
size=0.05em,inner sep=1pt]}]
 \node[main_node] (1) at (0,0) {};
    \node[main_node] (2) at (-1,1)  {};
    \node[main_node] (3) at (2,1) {};
    \node[main_node] (4) at (1,0) {};
    \node[main_node] (5) at (-1,0) {};
    \node[main_node] (6) at (2,0) {};
 \draw[main_node]  (1) edge node{$3\ \ \ \ $} (2);   
 \draw[main_node]  (2) edge node{$\begin{array}{c}\alpha\\ {}\end{array}$} (3);  
 \draw[main_node]  (3) edge node{$\ \ \ \ 2$} (4);    
 \draw[main_node]  (5) edge node{ } (1);    
 \draw[main_node]  (6) edge node{ } (4);    
 \draw[main_node]  (1) edge node{$\begin{array}{c}1\\ {}\end{array}$} (4);
\end{tikzpicture}

  \end{center}
corresponds to the tropical transfer matrix 
\[ A=\left(\begin{array}{cc}1&3\\4&\max(6,\alpha)\end{array}\right)\enspace ,\]
where the source nodes (at left) and the target nodes (at right) are numbered from bottom to top.
\end{exa}

A totally connected planar network with $n$ source and target nodes is shown in \Cref{n}. The diagonal arcs, and the horizontal arcs in the middle of the
network, are equipped with weights $w_{i,j}$. 
 We refer to the matrix $(w_{i,j})$ as \textit{matrix of weights}. 
When a weight of an arc
is not shown on this graph, this weight will be interpreted as unitary
(i.e., weight $1$ if classical weights are considered,
and weight $0$ if tropical weights are considered).
This weighted planar network arises classically
in the parametrization
of totally positive matrices~\cite[Figure~2]{F&Z}.
We shall refer to this as the {\em canonical} totally connected planar network over $n$ sources and $n$ targets,  denoted  by $G_n$.

If the weights $w_{ij}$ are thought of as indeterminates, then, a path
from a source to a target node is determined uniquely by
its weight or by its tropical weight. For instance, the unique path with
tropical weight $w_{n,n-1}+ w_{n-1,n-1}+ w_{n-1,n}$
is shown by the dashed line in blue in~\Cref{n}.

\section{Parametrization of tropical totally positive matrices by planar networks}\label{stpnnew}

In this section, we establish our main result, \Cref{PN},
showing that a tropical totally positive matrix can be
represented in a unique way as the transfer matrix
of a weighted planar network.
\begin{df}
Let~$G_n$ be the canonical totally connected planar network
with weights $w_{ij}$, as in Figure~\ref{n}.  
The \textit{uppermost} path from $i$ to $j$ in $G$ is
the path with tropical weight
$$\UM_{i,j}=
\begin{cases}
\sum_{t=i}^j w_{i,t},&\ i\leq j\\
\sum_{t=j}^i w_{t,j},&\ i> j.
\end{cases}$$
\end{df}
The notion of uppermost path is illustrated in~\Cref{n}.
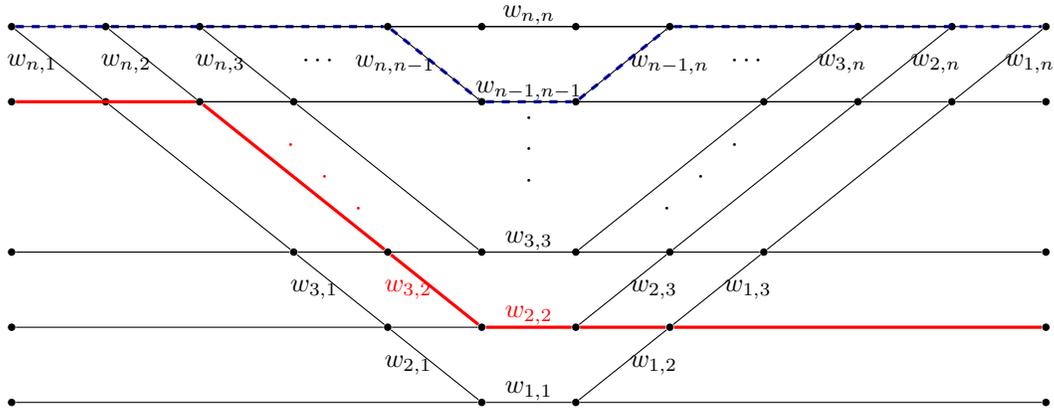
\begin{figure}[htbp]\begin{tikzpicture}[main_node/.style={circle,fill=black,
minimum size=0.05em,inner sep=1pt]}]
    \node[main_node] (31) at (-6.25,1) {};
    \node[main_node] (32) at (-6.25,0) {};
    \node[main_node] (30) at (-6.25,2) {};
    \node[main_node] (29) at (-6.25,4) {};
    \node[main_node] (15) at (-6.25,5) {};
    \node[main_node] (16) at (-5,5) {};
    \node[main_node] (11) at (-5,4) {};
    \node[main_node] (12) at (-3.75,4) {};
    \node[main_node] (25) at (-3.75,5) {};
    \node[main_node] (7) at (-2.5,2) {};
    \node[main_node] (21) at (-2.5,4) {};
    \node[main_node] (26) at (-1.25,5) {};
    \node[main_node] (2) at (-1.25,1)  {};
    \node[main_node] (8) at (-1.25,2) {};
    \node[main_node] (1) at (0,0) {};
    \node[main_node] (5) at (0,1) {};
    \node[main_node] (19) at (0,2) {};
   \node[main_node] (22) at (0,4) {};
    \node[main_node] (37) at (0,5) {};    
    \node[main_node] (38) at (1.25,5) {};
    \node[main_node] (4) at (1.25,0) {};
    \node[main_node] (24) at (1.25,4) {};
    \node[main_node] (20) at (1.25,2) {};
     \node[main_node] (6) at (1.25,1) {};
     \node[main_node] (3) at (2.5,1) {};
    \node[main_node] (27) at (2.5,5) {};
     \node[main_node] (9) at (2.5,2) {};
    \node[main_node] (10) at (3.75,2) {};
    \node[main_node] (23) at (3.75,4) {};
    \node[main_node] (13) at (5,4) {};
     \node[main_node] (28) at (5,5) {};
    \node[main_node] (17) at (6.25,5) {};
    \node[main_node] (14) at (6.25,4) {};
    \node[main_node] (18) at (7.5,5) {};
    \node[main_node] (33) at (7.5,4) {};
    \node[main_node] (34) at (7.5,2) {};
    \node[main_node] (35) at (7.5,1) {};
    \node[main_node] (36) at (7.5,0) {};
    \draw[color=blue,dashed,very thick] (15) --(26);
    \draw[color=blue,dashed,very thick] (26) --(22);
    \draw[color=blue,dashed,very thick] (22) --(24);
    \draw[color=blue,dashed,very thick] (24) --(27);
                    \draw[color=blue,dashed,very thick] (27) --(18);
 \draw[main_node]  (1) edge node{$\begin{array}{c}w_{1,1}\\ {}\end{array}$} (4);
 \draw[main_node]  (1) edge node{$w_{2,1}\ \ \ \ \ \ $} (2);   
 \draw[main_node]  (3) edge node{$\ \ \ \ \ \ \ w_{1,2}$} (4);    
\draw[main_node,color=red,very thick]  (5) edge node{$\begin{array}{c}w_{2,2}\\ {}\end{array}$} (6);  
\draw[main_node]  (2) edge node{} (5);  
\draw[main_node,color=red,very thick]  (3) edge node{} (6);  
 \draw[main_node]  (2) edge node{$w_{3,1}\ \ \ \ \ \ $} (7);
 \draw[main_node,color=red,very thick]  (5) edge node{$w_{3,2}\ \ \ \ \ \ $} (8);
 \draw[main_node]  (6) edge node{$\ \ \ \ \ \ \ w_{2,3}$} (9);
 \draw[main_node]  (3) edge node{$\ \ \ \ \ \ \ w_{1,3}$} (10);
\draw[main_node]  (7) edge node{} (10);  
\draw[main_node]  (8) edge node{$\begin{array}{c}w_{3,3}\\ {}\end{array}$} (9);  
\draw[main_node]  (12) edge node{$\begin{array}{c}{}\\{}\\w_{n-1,n-1}\\ \cdot\\\cdot\\\cdot\end{array}$} (13);  
\draw[main_node]  (16) edge node{$\begin{array}{c}w_{n,n}\\ {}\end{array}$} (17);  
\draw[main_node]  (11) edge node{$w_{n,1}\ \ \ \ \ \ $} (15);  
\draw[main_node]  (12) edge node{$w_{n,2}\ \ \ \ \ \ $} (16);  
\draw[main_node]  (21) edge node{$w_{n,3}\ \ \ \ \ \ $} (25);  
\draw[main_node]  (22) edge node{$\cdots\ \ w_{n,n-1}\ \ \ \ \ \ \ \ \ \ \ \ \ \ \ $} (26);  
\draw[main_node]  (14) edge node{$\ \ \ \ \ \ \ w_{1,n}$} (18);  
\draw[main_node]  (13) edge node{$\ \ \ \ \ \ \ w_{2,n}$} (17);  
\draw[main_node]  (23) edge node{$\ \ \ \ \ \ \ w_{3,n}$} (28);  
\draw[main_node]  (24) edge node{$\ \ \ \ \ \ \ \ \ \ \ \ \ \ \ \ \ w_{n-1,n}\ \ \cdots$} (27);  
\draw[main_node]  (11) edge node{} (14);  
\draw[main_node,color=red,very thick]  (29) edge node{} (12);  
\draw[main_node]  (15) edge node{} (18); 
\draw[main_node]  (7) edge node{} (11);  
\draw[main_node,color=red,very thick]  (8) edge node{$\ \ \ \ \ \ \ \begin{array}{ccc}\cdot &{}&{}\\{}&\cdot &{}\\{}&{}&\cdot \end{array}$} (12);  
\draw[main_node]  (19) edge node{} (21);  
\draw[main_node]  (20) edge node{$\ \ \ \ \ \ \ \begin{array}{ccc}{} &{}&\cdot\\{}&\cdot &{}\\ \cdot&{}&{} \end{array}$} (23);  
\draw[main_node]  (9) edge node{} (13);  
\draw[main_node]  (10) edge node{} (14);   

 \draw[main_node]  (7) edge node{} (30);  
\draw[main_node]  (2) edge node{} (31);   
\draw[main_node]  (1) edge node{} (32);   
\draw[main_node]  (14) edge node{} (33);   
\draw[main_node]  (10) edge node{} (34);   
\draw[main_node,color=red,very thick]  (3) edge node{} (35);   
\draw[main_node]  (4) edge node{} (36);    
\end{tikzpicture}
\caption{The canonical totally connected planar network $G_n$. The uppermost path~$U_{n-1,2}$ from source node $n-1$ to target node $2$ is shown in  red (bold line). The path with tropical weight $w_{n,n-1}+w_{n-1,n-1}+w_{n-1,n}$ is shown in blue (dashed line).}
\label{n}
\end{figure}

\begin{df} \label{d-trapeze}
  We say that the weights of $G_n$ satisfy the (strict)
  \textit{trapeze inequality} if
$$w_{i,i}>w_{i,i-1}+w_{i-1,i-1}+w_{i-1,i},\ \forall i=2,\dots n.$$
 This is illustrated  in \Cref{fig-trapeze}.  Note that the trapezes occurring in these inequalities are  central, between consecutive levels, and minimal.
 \begin{figure}[htbp]
 \begin{center}
  \begin{tikzpicture}[main_node/.style={circle,fill=black,minimum 
size=0.05em,inner sep=1pt]}]
 \node[main_node] (1) at (0,0) {};
    \node[main_node] (2) at (-1,1)  {};
    \node[main_node] (3) at (2,1) {};
    \node[main_node] (4) at (1,0) {};
    \node[main_node] (5) at (-1,0) {};
    \node[main_node] (6) at (2,0) {};
    \draw[main_node,color=blue,very thick,left,dashed]  (1) edge node{$ w_{i,i-1}$} (2);
    \node (dummy) at (0.5,0.5) {$\vee$};
 \draw[main_node,color=red,very thick,above]  (2) edge node{$\begin{array}{c}w_{i,i}\end{array}$} (3);  
 \draw[main_node,color=blue,very thick,right,dashed]  (3) edge node{$\ w_{i-1,i}$} (4);    
 \draw[main_node]  (5) edge node{ } (1);    
 \draw[main_node]  (6) edge node{ } (4);    
 \draw[main_node,color=blue,very thick,below,dashed]  (1) edge node{$w_{i-1,i-1}$} (4);
  \end{tikzpicture}
 \end{center}
 \caption{The trapeze inequality. The (tropical) weight of the path in blue (dashed path) is dominated by the tropical weight of the path in red (bold path) with the same source and target. }\label{fig-trapeze}
 \end{figure}
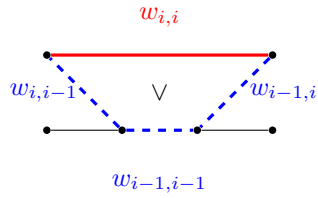

\end{df}
\begin{df}\label{d-para}
  We say that the weights of $G_n$ satisfy the (strict) \textit{parallelogram inequality}
if 
$$w_{i,1}<w_{i,2}<\dots<w_{i,i-2}<w_{i,i-1}\ \text{ and }\ w_{1,i}<w_{2,i}<\dots<w_{i-2,i}<w_{i-1,i},\ \forall i=2,\dots n.$$
In other words, the weights are increasing on every level, as one moves from the left or from the right of the planar network towards the center. This is illustrated in \Cref{fig-para}. 

We shall say that the weights satisfy the {\em weak} trapeze or parallelogram inequalities if the conditions in Definitions \ref{d-trapeze} and \ref{d-para} involving strict inequalities are replaced by non strict ones.

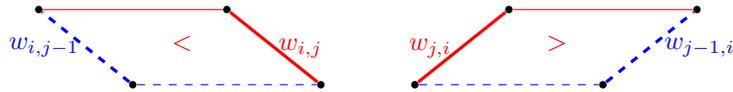
\begin{figure}[htbp]
\begin{center}
  \begin{tikzpicture}[main_node/.style={circle,fill=black,
minimum size=0.05em,inner sep=1pt]}]
    \node[main_node] (25) at (-3.75,5) {};
    \node[main_node] (21) at (-2.5,4) {};
    \node[main_node] (26) at (-1.25,5) {};
   \node[main_node] (22) at (0,4) {};
    \node[main_node] (24) at (1.25,4) {};
    \node[main_node] (27) at (2.5,5) {};
    \node[main_node] (23) at (3.75,4) {};
     \node[main_node] (28) at (5,5) {};
\draw[main_node,color=blue,very thick,left,dashed]  (21) edge node{$w_{i,j-1}$} (25);  
\draw[main_node,color=red,very thick]  (22) edge node{$<\ \ \ \ \ \ \ \ \ w_{i,j}\ \ \ \ \ \ $} (26);  
\draw[main_node,color=blue,very thick,right,dashed]  (23) edge node{$\ w_{j-1,i}$} (28);  
\draw[main_node,color=red,very thick]  (24) edge node{$\ \ \ \ \ \ w_{j,i}\ \ \ \ \ \ \ \ \ \ >$} (27);  
\draw[main_node,color=blue,dashed]  (24) edge node{} (23); 
\draw[main_node,color=blue,dashed]  (21) edge node{} (22); 
\draw[main_node,color=red]  (28) edge node{} (27); 
\draw[main_node,color=red]  (26) edge node{} (25); 
\end{tikzpicture}
\end{center}
\caption{The parallelogram inequalities. The (tropical) weight of the path in blue (dashed path) is dominated by the weight of the path in red with the same endpoints.}\label{fig-para}
\end{figure}
\end{df}

\begin{pro}\label{rltn}
  The weights of the
  canonical totally connected planar network $G_n$ (in Figure~\ref{n})
  satisfy the  weak (resp.\  strict) trapeze and parallelogram inequalities
if and only if the uppermost path is a path (resp.\ the  only path)
  of maximal tropical weight from $i$ to $j$ in $G_n$, for every $i,j$.\end{pro}

\begin{proof} We prove only the equivalence in the case of strict inequalities. The case of  weak inequalities is similar.
 
\underline{$\Rightarrow$ :}
  Consider an arbitrary path $\pi$ from a source $i$ to a target $j$.
  The strict trapeze and parallelogram inequalities allow us to perform
  mutation operations, which replace a path $\pi$ by another path with
  the same endpoints and with a larger (tropical) weight.
  For instance, if the path $\pi$ contains as a subpath
  the path with weight $w_{i,i-1}+ w_{i-1,i-1}+ w_{i-1,i}$
  shown in \Cref{fig-trapeze} (in dashed blue), replacing this subpath by the path with weight $w_{i,i}$ (in red) yields a path $\pi'$ with the announced
  property. Similarly, if $\pi$ contains as a subpath one of the paths
  in dashed blue in the parallelograms shown in \Cref{fig-para}, replacing this subpath by the opposite path in the parallelogram (in red), we end up again
  with a path $\pi'$ with the announced properties. Carrying out these
  operations until no mutation of the path is possible, we arrive
  at a path from $i$ to $j$, which is necessarily a uppermost path,
  and whose weight dominates the weight of $\pi$.
  This mutation procedure is illustrated in \Cref{fig-mutation}.

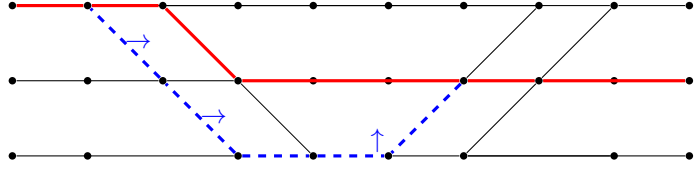
\begin{figure}[htbp]\begin{tikzpicture}[main_node/.style={circle,
fill=black,minimum size=0.05em,inner sep=1pt]}]
    \node[main_node] (1) at (-0.5,0) {};
    \node[main_node] (2) at (-1.5,1)  {};
    \node[main_node] (3) at (3.5,1) {};
    \node[main_node] (4) at (2.5,0) {};
    \node[main_node] (41) at (0.5,0) {};
    \node[main_node] (44) at (1.5,0) {};
\node[main_node] (42) at (0.5,1) {};
    \node[main_node] (43) at (1.5,1) {};
\node[main_node] (45) at (0.5,2) {};
    \node[main_node] (46) at (1.5,2) {};
    \node[main_node] (5) at (-0.5,1) {};
    \node[main_node] (6) at (2.5,1) {};
    \node[main_node] (7) at (-2.5,2) {};
    \node[main_node] (8) at (-1.5,2) {};
    \node[main_node] (9) at (3.5,2) {};
    \node[main_node] (10) at (4.5,2) {};
    \node[main_node] (11) at (-0.5,2) {};
    \node[main_node] (12) at (2.5,2) {};
    \node[main_node] (23) at (-3.5,2) {};
    \node[main_node] (13) at (-3.5,0) {};
    \node[main_node] (14) at (4.5,0) {};
    \node[main_node] (15) at (-3.5,1) {};
    \node[main_node] (16) at (4.5,1) {};
    \node[main_node] (22) at (-2.5,1) {};
    \node[main_node] (21) at (-2.5,0) {};
    \node[main_node] (31) at (5.5,2) {};
    \node[main_node] (32) at (5.5,1) {};
    \node[main_node] (33) at (5.5,0) {};
 \draw[main_node,color=blue,dashed,very thick]  (1) edge (41);

 \draw[main_node,color=blue,dashed,very thick]  (41) edge node{$\begin{array}{ccc}&&\uparrow\\&&\end{array}$} (44);
 \draw[main_node,color=blue,dashed,very thick]  (1) edge node{$\ \ \ \rightarrow$} (2);   
 \draw[main_node,color=red,very thick]  (6) edge node{$\ \ \ $} (5);   
 \draw[main_node]  (3) edge node{$\ \ \ \ \ \ \ $} (4);    

\draw[main_node]  (2) edge node{} (5);  
\draw[main_node]  (4) edge node{} (44);  
\draw[main_node]  (3) edge node{} (6);  
 \draw[main_node,color=blue,dashed, very thick]  (2) edge node{$\ \ \ \rightarrow$} (7);
 \draw[main_node,color=red,very thick]  (5) edge node{$\ \ \ $} (8);

 \draw[main_node]  (6) edge node{$\ \ \ \ \ \ \  $} (9);
 \draw[main_node]  (3) edge node{$\ \ \ \ \ \ \ $} (10);
\draw[main_node]  (7) edge node{} (10);  
\draw[main_node,color=red,very thick]  (7) edge node{} (8);  
\draw[main_node]  (13) edge node{} (1);
  \draw[main_node]  (14) edge node{} (4);
  \draw[main_node]  (15) edge node{} (2);
  \draw[main_node]  (16) edge node{} (3);  
  \draw[main_node,color=red,very thick]  (23) edge node{} (7);  
  \draw[main_node]  (10) edge node{} (31);  
  \draw[main_node,color=red,very thick]  (3) edge node{} (32);  

  \draw[main_node]  (4) edge node{} (33);  
\draw[main_node]  (8) edge node{$\begin{array}{c}\\ {}\end{array}$} (9);  

\draw[main_node,color=blue,dashed,very thick]  (6) edge node{} (44);  
\draw[main_node,color=red,very thick]  (3) edge node{} (6);  

 \draw[main_node,very thin]  (41) edge node{} (5);   
  \end{tikzpicture}
  \caption{Applying two parallelogram mutations followed by one trapeze mutation to the path $3\to 1$ (dashed blue) yields a uppermost path with the same endpoints and a larger (tropical) weight (bold red).}\label{fig-mutation}
\end{figure}

\underline{$\Leftarrow$ :} 
  Our assumption entails that
  the uppermost path with weight $w_{i,i}$  is the  only  path of maximal weight from $i$ to $i$, for every $i$, and therefore 
  $$\UM_{i,i}=w_{i,i}>w_{i,i-1}+w_{i-1,i-1}+w_{i-1,i},\ \forall i,$$
  showing that the strict trapeze inequality holds.

  Let us now assume that $i<j$.
  The uppermost path $\UM_{i,j}=\sum_{t=i}^j w_{i,t}$
  is the  only  maximal path from $i$ to $j$. 
Therefore 
$$\UM_{i,j}=\sum_{t=i}^j w_{i,t}=w_{i,j}+\sum_{t=i}^{j-1} w_{i,t}>w_{i-1,j}+\sum_{t=i}^{j-1} w_{i,t}=w_{i-1,j}+\UM_{i,j-1}\ \ \forall i<j,$$   which implies 
$w_{i,j}>w_{i-1,j}\ \forall i<j$, showing that the second strict parallelogram inequality in \Cref{fig-para} holds. The first of these strict parallelogram
inequalities is obtained by a symmetrical argument, assuming
that $i>j$.

\end{proof}

In the next definition, the map $\psi$ is constructed so that $\psi(W)_{ij}$
is the weight of the uppermost path from $i$ to $j$. We shall see
that $\psi$ is invertible, and provide an explicit expression
for its inverse, $\phi$.

\begin{df}
The map $\psi:\R^{n\times n}\mapsto\R^{n\times n}$ is defined by
$$\psi(W)=(\psi(W)_{i,j})\ :\ 
\psi(W)_{i,j}=\begin{cases}
\sum_{t=i}^j w_{i,t},&\ i\leq j\\
\sum_{t=j}^i w_{t,j},&\ i> j\enspace ,
\end{cases}$$
for  every $n\times n$ matrix $W=(w_{i,j})$.
The map $\phi:\R^{n\times n}\mapsto\R^{n\times n}$ is defined by
$$\phi(A)=(\phi(A)_{i,j})\ :\ 
\phi(A)_{i,j}=\begin{cases}
a_{i,j},&\ i=j\\
a_{i,j}-a_{i,j-1},
&\ i< j,\\ 
a_{i,j}-a_{i-1,j},&  i>j\enspace ,\end{cases}$$ 
for  every $n\times n$ matrix $A=(a_{i,j})$.
\end{df}

\begin{pro}\label{ID}
The maps $\phi$ and $\psi$ satisfy
$\psi\circ\phi(A)=A$ and~$\phi\circ\psi(W)=W$
for all matrices $A,W\in\R^{n\times n}$.
\end{pro}
\begin{proof}
  Suppose that $A=\psi(W)$.
  Then,
  $\phi(A)_{i,i}=a_{i,i}=w_{ii}$.
  Moreover, for $i<j$,
  $\phi(A)_{i,j}=a_{i,j}-a_{i,j-1}
  =\sum_{t=i}^j w_{i,t}-\sum_{t=i}^{j-1} w_{i,t}= 
  w_{i,j}$. A dual
  calculus applies to the case $i>j$. We deduce
  that $\phi\circ\psi(W)=W$.
  
  Suppose now that $W=\phi(A)$.
  Then, for $i\leq j$,
  $\psi(W)_{ij}=
  \sum_{t=i}^j w_{i,t}
  =w_{i,i}+
  \sum_{t=i+1}^j w_{i,t}$.
  If $i=j$, we deduce that $\psi(W)_{ii}=w_{ii}=a_{ii}$.
  If $i<j$, we get the telescopic sum
  $\psi(W)_{ij}=a_{i,i}+ \sum_{t=i+1}^i(a_{i,t}-a_{i,t-1})=a_{i,j}$.
  The case in which $j>i$ is dual. We deduce that
  $\psi\circ\phi(A)=A$.
  
\end{proof}
We shall   need the following observation.

\begin{pro}\label{phipsi}
  If $A\in \TN^\trop(\R)$ (resp.\ $A\in \TP^\trop$),
  the weights $W:=\phi(A)$ satisfy the weak (resp.\ strict)
  trapeze
  and parallelogram inequalities.
\end{pro}
\begin{proof}
  Suppose that $W=\phi(A)$, where $A\in\TN^\trop(\R)$.

  Then, $w_{i,i}-w_{i,i-1}-w_{i-1,i-1}-w_{i-1,i}
  =a_{i,i}-(a_{i,i-1}-a_{i-1,i-1})-a_{i-1,i-1}-(a_{i-1,i}-a_{i-1,i-1})
  =a_{i,i}+a_{i-1,i-1}-a_{i-1,i}-a_{i,i-1}\geq 0$, because
  the principal minor of the matrix $A$ with indices
  $\{i-1,i\}$ is tropically nonnegative.

  Suppose now that $i> j$. Then,
  $w_{i,j}-w_{i,j-1}=a_{i,j}-a_{i-1,j}-(a_{i,j-1}-a_{i-1,j-1})
  \geq 0$, because the $\{i-1,i\}\times \{j-1,j\}$
  minor of $A$ is is tropically nonnegative.

  A dual argument applies to the situation
  in which $i<j$. We conclude that $W$ satisfies the weak trapeze
  and parallelogram inequalities. 
  If $A\in \TP^\trop$, we derive the strict trapeze
  and parallelogram inequalities along the same lines.
  \end{proof}
The following proposition holds for any totally connected planar network (not only for the canonical one). It is deduced in~\cite{GN17} from the classical result showing that the ordinary transfer matrix of a totally connected planar network is totally positive, see e.g.~ \cite[Corollary 2]{F&Z}. Observe that in the tropical setting, we only get total {\em nonnegativity}. This is because when
applying the nonarchimedean valuation to a strict inequality, we only
derive a weak inequality.
\begin{pro}[{\cite[Corollary~7.9]{GN17}}]
  \label{PNTN}
  The tropical transfer matrix of a planar network $G$ is tropical totally nonnegative. In particular if $G$ is totally 
connected then its transfer matrix is in $\TN^\trop(\R)$.
\end{pro}

In the following theorem,
 considering the canonical totally connected
 planar network~$G_n$ shown in \Cref{n},
 we refine the result of Proposition~\ref{PNTN}  
by characterizing the network weights that corresponds to
tropical totally {\em positive} matrices.

\begin{thm}\label{PN}Let~$G_n$ be the planar network
with weights $w_{ij}$, as in Figure~\ref{n}. 
\begin{enumerate}

\item  If for every~$i,j\in[n]$ the  only  path of maximal weight in~$G_n$ is $\UM_{i,j}$,  then the  transfer
 matrix of $G_n$ is in~$\TP^\trop$.

\item\label{(2)} For every square matrix~$ A\in \TN^\trop(\mathbb{R})$ there exists a 
choice of  matrix of weights $W=(w_{ij})$, such that~$A$ is the transfer matrix
of $W$. This choice becomes unique if we require in addition that the weights $w_{i,j}$ satisfy the weak parallelogram and trapeze inequalities.

\item If~$A\in \TP^\trop$ then there exists a unique choice of  matrix of weights $W$ such that $A$ is the transfer matrix of $W$.

\end{enumerate}
\end{thm}
\begin{proof}

Denote by~$A=(a_{i,j})$  the transfer matrix of~$G_n$.

(1) Using Proposition~\ref{rltn}, the parallelogram and trapeze inequalities hold.
By Theorem~3.4 and Lemma~3.2 of~\cite{GN17}, the tropical total positivity of $A$ is equivalent to the strict inequalities
\begin{align}
  a_{i,j}+ a_{i-1,j-1}>a_{i-1,j}+ a_{i,j-1}\ \forall i,j\in[n]
  \enspace.\label{e-s-tocheck}
  \end{align}
By assumption, $a_{i-1,j-1}=\UM_{i-1,j-1}$, $a_{i,j-1}=\UM_{i,j-1}$.
Assume first that $i< j$. Then, considering the shape of uppermost paths,
we have $a_{i-1,j}=\UM_{i-1,j}=\UM_{i-1,j-1}+w_{i-1,j}$,
and similarly, $a_{i,j}=\UM_{i,j-1}+w_{i,j}$.
Using the parallelogram inequality,
we deduce that $a_{i,j}+ a_{i-1,j-1}- a_{i-1,j}+ a_{i,j-1}
= w_{ij}-w_{i-1,j}>0$, so that~\eqref{e-s-tocheck} holds
when $i<j$.
The situation in which $i>j$ is dual. 
If $i=j$, the identity $a_{ij}=\UM_{i,j-1}+w_{i,j}$
is replaced by $a_{ij}=a_{ii}=\UM_{ii}$.
We deduce that
$a_{i,j}+ a_{i-1,j-1}- a_{i-1,j}+ a_{i,j-1}
=\UM_{i,i}-\UM_{i,i-1}-w_{i-1,i}
= w_{ii}-w_{i,i-1}-w_{i-1,i-1}-w_{i-1,1}>0$
by the trapeze inequality, showing that~\eqref{e-s-tocheck}
holds in all cases.

(2) Consider $A\in \TN^\trop(\R)$, then by \Cref{phipsi} the matrix $W:=\phi(A)$ satisfies the weak trapeze and 
parallelogram inequalities. By~\Cref{rltn}, the $i,j$ entry of the transfer matrix $B$ arising  from the matrix of weights $W$   coincides with the weight 
$\UM_{i,j}$ of the uppermost paths. So $B=\psi(W)$ by construction of $\psi$. We showed in \Cref{ID} that $\psi\circ\phi=\Id$, it follows that $A=B=\psi(W)$.

Suppose now that $A=\psi(W')$ where $W'$ is a matrix of weights that satisfies the weak parallelogram and trapeze inequalities. Then by \Cref{ID}, 
$W'=\phi(\psi(W'))=\phi(A)$, showing that $W'$ is uniquely determined by $A$.

(3) We show that if  $A=\psi(W)\in\TP^\trop$, then the weight $w_{i,j}$ satisfy the strict trapeze and parallelogram inequalities.
Figure~\ref{n'} is a modification of Figure~\ref{n}  in which we introduce  an additional  node in every level, denoted by $d_i,\ \forall i\in[n]$.

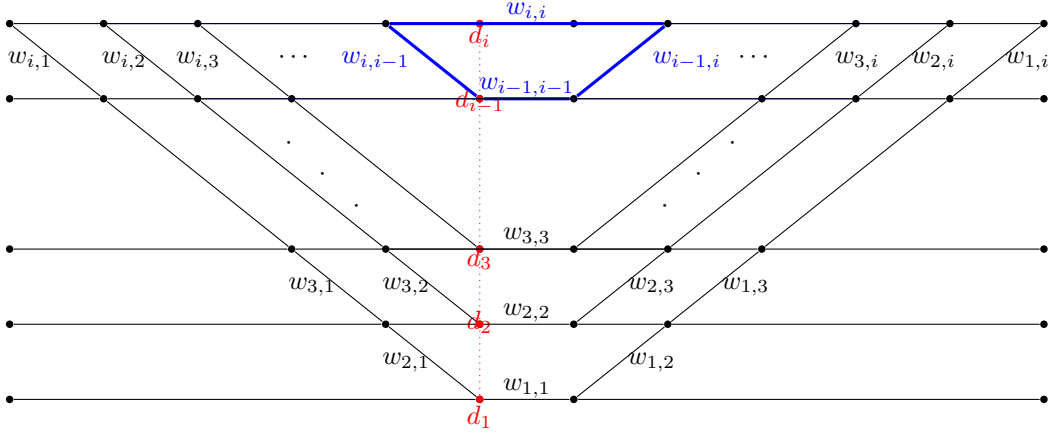
\begin{figure}[htbp]\begin{tikzpicture}[main_node/.style={circle,fill=black,
minimum size=0.05em,inner sep=1pt]}]
    \node[main_node] (31) at (-6.25,1) {};
    \node[main_node] (32) at (-6.25,0) {};
    \node[main_node] (30) at (-6.25,2) {};
    \node[main_node] (29) at (-6.25,4) {};
    \node[main_node] (15) at (-6.25,5) {};
    \node[main_node] (16) at (-5,5) {};
    \node[main_node] (11) at (-5,4) {};
    \node[main_node] (12) at (-3.75,4) {};
    \node[main_node] (25) at (-3.75,5) {};
    \node[main_node] (7) at (-2.5,2) {};
    \node[main_node] (21) at (-2.5,4) {};
    \node[main_node] (26) at (-1.25,5) {};
    \node[main_node] (2) at (-1.25,1)  {};
    \node[main_node] (8) at (-1.25,2) {};
    \node[main_node] (1) at (0,0) {};
    \node[main_node] (5) at (0,1) {};
    \node[main_node] (19) at (0,2) {};
   \node[main_node] (22) at (0,4) {};
    \node[main_node] (37) at (0,5) {};    
    \node[main_node] (38) at (1.25,5) {};
    \node[main_node] (4) at (1.25,0) {};
    \node[main_node] (24) at (1.25,4) {};
    \node[main_node] (20) at (1.25,2) {};
     \node[main_node] (6) at (1.25,1) {};
     \node[main_node] (3) at (2.5,1) {};
    \node[main_node] (27) at (2.5,5) {};
     \node[main_node] (9) at (2.5,2) {};
    \node[main_node] (10) at (3.75,2) {};
    \node[main_node] (23) at (3.75,4) {};
    \node[main_node] (13) at (5,4) {};
     \node[main_node] (28) at (5,5) {};
    \node[main_node] (17) at (6.25,5) {};
    \node[main_node] (14) at (6.25,4) {};
    \node[main_node] (18) at (7.5,5) {};
    \node[main_node] (33) at (7.5,4) {};
    \node[main_node] (34) at (7.5,2) {};
    \node[main_node] (35) at (7.5,1) {};
    \node[main_node] (36) at (7.5,0) {};

    \node[main_node,color=red] (37) at (0,0) {};
    \node[main_node,color=red] (38) at (0,5) {};
\draw[main_node,color=red,dotted]  (37) edge node{$\begin{array}{c}\\d_i\\\\d_{i-1}\\ \\ \\\\\\d_3\\\\d_2\\\\\\d_1\end{array}$} (38);  
    \node[main_node,color=red] (39) at (0,1) {};
    \node[main_node,color=red] (40) at (0,4) {};
    \node[main_node,color=red] (41) at (0,2) {};

\draw[main_node,color=blue,very thick]  (27) edge node{} (26);  
\draw[main_node,color=blue,very thick]  (24) edge node{} (22);

 \draw[main_node]  (1) edge node{$\begin{array}{c}w_{1,1}\\ {}\end{array}$} (4);
 \draw[main_node]  (1) edge node{$w_{2,1}\ \ \ \ \ \ $} (2);   
 \draw[main_node]  (3) edge node{$\ \ \ \ \ \ \ w_{1,2}$} (4);    
\draw[main_node]  (5) edge node{$\begin{array}{c}w_{2,2}\\ {}\end{array}$} (6);  
\draw[main_node]  (2) edge node{} (5);  
\draw[main_node]  (3) edge node{} (6);  
 \draw[main_node]  (2) edge node{$w_{3,1}\ \ \ \ \ \ $} (7);
 \draw[main_node]  (5) edge node{$w_{3,2}\ \ \ \ \ \ $} (8);
 \draw[main_node]  (6) edge node{$\ \ \ \ \ \ \ w_{2,3}$} (9);
 \draw[main_node]  (3) edge node{$\ \ \ \ \ \ \ w_{1,3}$} (10);
\draw[main_node]  (7) edge node{} (10);  
\draw[main_node]  (8) edge node{$\begin{array}{c}w_{3,3}\\ {}\end{array}$} (9);  
\draw[main_node,color=blue]  (12) edge node{$\begin{array}{c}{}\\{}\\w_{i-1,i-1}\\\\\\\\ \end{array}$} (13);  
\draw[main_node,color=blue]  (16) edge node{$\begin{array}{c}w_{i,i}\\ {}\end{array}$} (17);  
\draw[main_node]  (11) edge node{$w_{i,1}\ \ \ \ \ \ $} (15);  
\draw[main_node]  (12) edge node{$w_{i,2}\ \ \ \ \ \ $} (16);  
\draw[main_node]  (21) edge node{$w_{i,3}\ \ \ \ \ \ \cdots$} (25);  
\draw[main_node,color=blue,very thick]  (22) edge node{$\ \ w_{i,i-1}\ \ \ \ \ \ \ \ \ \ \ \ \ \ \ $} (26);  
\draw[main_node]  (14) edge node{$\ \ \ \ \ \ \ w_{1,i}$} (18);  
\draw[main_node]  (13) edge node{$\ \ \ \ \ \ \ w_{2,i}$} (17);  
\draw[main_node]  (23) edge node{$\cdots\ \ \ \ \ \ \ w_{3,i}$} (28);  
\draw[main_node,color=blue,very thick]  (24) edge node{$\ \ \ \ \ \ \ \ \ \ \ \ \ \ \ \ \ w_{i-1,i}\ \ $} (27);  
\draw[main_node]  (11) edge node{} (14);  
\draw[main_node]  (15) edge node{} (18); 
\draw[main_node]  (7) edge node{} (11);  
\draw[main_node]  (8) edge node{$\ \ \ \ \ \ \ \begin{array}{ccc}\cdot &{}&{}\\{}&\cdot &{}\\{}&{}&\cdot \end{array}$} (12);  
\draw[main_node]  (19) edge node{} (21);  
\draw[main_node]  (20) edge node{$\ \ \ \ \ \ \ \begin{array}{ccc}{} &{}&\cdot\\{}&\cdot &{}\\ \cdot&{}&{} \end{array}$} (23);  
\draw[main_node]  (9) edge node{} (13);  
\draw[main_node]  (10) edge node{} (14);   
\draw[main_node]  (11) edge node{} (29);  
 \draw[main_node]  (7) edge node{} (30);  
\draw[main_node]  (2) edge node{} (31);   
\draw[main_node]  (1) edge node{} (32);   
\draw[main_node]  (14) edge node{} (33);   
\draw[main_node]  (10) edge node{} (34);   
\draw[main_node]  (3) edge node{} (35);   
\draw[main_node]  (4) edge node{} (36);    
\end{tikzpicture}
\caption{Proof of \Cref{PN}, part (3),(i). The trapeze inequality fails: $w_{i,i}\leq w_{i,i-1}+w_{i-1,i-1}+w_{i-1,i}$}
\label{n'}
\end{figure}

(i). Assume that the strict trapeze inequality fails. That means that there exists minimal $i$ such that $$w_{i,i}\leq w_{i,i-1}+w_{i-1,i-1}+w_{i-1,i},$$ and the 
paths going from $i$ to $i$ will therefore pass through $d_{k},$ for some $ k< i$. Consider the transfer matrices $Z^\ell,Z^r$ of the sub graphs on the left 
and right  to the red line in Figure~\ref{n'} respectively. We get $A_{[i],[i]}=Z^\ell_{[i],[i-1]} \odot Z^r_{[i-1],[i]}$, where $M_{I,J}$ denotes the $I\times J$ submatrix of a matrix $M$ and recalling that $[i]=\{1,2,...,i\}$. 

We recall the following classical result in tropical linear algebra: {\em suppose that $F\in\rmax^{m\times m}$ can be factored as $F=G\odot H$ where $G\in\rmax^{m\times s}$, $H\in\rmax^{s\times m}$ and
 $s< m$,
then $F$ is tropically singular.} Several proofs of this property have appeared (in fact, as part of stronger results on ranks of tropical matrices), for instance see~\cite[Theorem~10.0.1]{G.Ths}, \cite[Theorem~1.4]{DSS} and \cite[Theorem~8.6]{LDTS}.
As a result, we get that $A_{[i],[i]}$ is tropically singular, contradicting that $A\in\TP^\trop$.

(ii). Next, we assume that the strict parallelogram inequality fails. That means that there exists $i\ne j$ s.t.~$w_{j,i}\leq w_{j-1,i}$ or 
$w_{j,i}\leq w_{j,i-1}$. Without loss of generality, we consider the former. Ordering all $w_{i,j}:\ i\ne j$ from right to   left and from bottom 
to top, we choose the first pair $i,j$ 
 that does  not satisfy the parallelogram inequality. That is, $w_{t,s}>w_{t',s'}\ \forall t'<t\text{ and }s'<s$ for every $t<i\text{ and }s<j$  
(see Figure~\ref{n'''}).
Calculating the $2\times 2$ minor of rows $j,j-1$ and columns $i,i-1$ of the transfer matrix, we get that the $(j-1,i-1)$,$(j-1,i)$ and $(j,i-1)$ entries are 
obtained by the  uppermost paths (also shown on Figure~\ref{n'''}). In particular $\UM_{j-1,i}=\UM_{j-1,i-1} +w_{j-1,i}$. However, the $j,i$ entry is obtained by the uppermost 
path from $j$ to $i-1$ concatenated by 
$w_{j-1,i}$. Therefore,
$$\left(\begin{array}{cc}
\UM_{j-1,i-1} & \UM_{j-1,i}\\
\UM_{j,i-1}  &\UM_{j,i-1} w_{j-1,i}
\end{array}\right)=\left(\begin{array}{cc}
\UM_{j-1,i-1} & \UM_{j-1,i-1} w_{j-1,i}\\
\UM_{j,i-1}  &\UM_{j,i-1} w_{j-1,i}
\end{array}\right)$$ which is singular. Contradiction.

\begin{figure}[htbp]\begin{tikzpicture}[main_node/.style={circle,fill=black,
minimum size=0.05em,inner sep=1pt]}]
    \node[main_node] (31) at (-6.25,1) {};
    \node[main_node] (30) at (-6.25,2) {};
    \node[main_node] (29) at (-6.25,4) {};
    \node[main_node] (15) at (-6.25,5) {};
    \node[main_node] (16) at (-5,5) {};
    \node[main_node] (11) at (-5,4) {};
    \node[main_node] (12) at (-3.75,4) {};
    \node[main_node] (25) at (-3.75,5) {};
    \node[main_node] (7) at (-2.5,2) {};
    \node[main_node] (21) at (-2.5,4) {};
    \node[main_node] (26) at (-1.25,5) {};
    \node[main_node] (2) at (-1.25,1)  {};
    \node[main_node] (8) at (-1.25,2) {};
  \node[main_node] (5) at (0,1) {};
    \node[main_node] (19) at (0,2) {};
   \node[main_node] (22) at (0,4) {};
    \node[main_node] (37) at (0,5) {};    
    \node[main_node] (38) at (1.25,5) {};
    \node[main_node] (24) at (1.25,4) {};
    \node[main_node] (20) at (1.25,2) {};
     \node[main_node] (6) at (1.25,1) {};
     \node[main_node] (3) at (2.5,1) {};
    \node[main_node] (27) at (2.5,5) {};
     \node[main_node] (9) at (2.5,2) {};
    \node[main_node] (10) at (3.75,2) {};
    \node[main_node] (23) at (3.75,4) {};
    \node[main_node] (13) at (5,4) {};
     \node[main_node] (28) at (5,5) {};
    \node[main_node] (17) at (6.25,5) {};
    \node[main_node] (14) at (6.25,4) {};
    \node[main_node] (18) at (7.5,5) {};

\node[draw=black,circle,white](1) at (8,5){$\color{black}i$};
\node[draw=black,circle,white](1) at (8,4){$\color{black}i-1$};
\node[draw=black,circle,white](1) at (8,2){$\color{black}j$};
\node[draw=black,circle,white](1) at (8,1){$\color{black}j-1$};

    \node[main_node] (33) at (7.5,4) {};
    \node[main_node] (34) at (7.5,2) {};
    \node[main_node] (35) at (7.5,1) {};
\draw[main_node,color=blue,dashed,very thick]  (5) edge node{$\begin{array}{c}w_{j-1,j-1}\\ {}\end{array}$} (6);  
\draw[main_node,color=blue,dashed,very thick]  (2) edge node{} (5);  
               \draw[main_node]  (3) edge node{} (6);  

 \draw[main_node]  (5) edge node{$<\ \ w_{j,j-1}\ \ \ \ \ \ $} (8);
 \draw[main_node,color=blue,dashed,very thick]  (6) edge node{$\ \ \ \ \ \ \ w_{j-1,j}\ \ >$} (9);

       \draw[main_node]  (7) edge node{} (10);

\draw[main_node,dashed,blue,very thick]  (8) edge node{} (20);  

\draw[main_node]  (12) edge node{$\begin{array}{c}{}\\{}\\w_{i-1,i-1}\\ \cdot\\\cdot\\\cdot\end{array}$} (13);  
\draw[main_node]  (16) edge node{$\begin{array}{c}w_{i,i}\\ {}\end{array}$} (17);  

\draw[main_node]  (12) edge node{$w_{i,j-1}\ \ \ \ \ \ $} (16);  
\draw[main_node]  (21) edge node{$w_{i,j}\ \ \ \ \ \ $} (25);  
\draw[main_node]  (22) edge node{$\cdots\ \ w_{i,i-1}\ \ \ \ \ \ \ \ \ \ \ \ \ \ \ $} (26);  

\draw[main_node,color=red,very thick]  (13) edge node{$\ \ \ \ \ \ \ \ \ \ \  w_{j-1,i}>\cdots>$} (17);  
\draw[main_node,color=blue,dashed,very thick]  (13) edge node{$\ \ \ \ \ \ \ \ \ \ \  w_{j-1,i}>\cdots>$} (17);  
\draw[main_node,color=red,very thick]  (23) edge node{$\ \ \ \ \ \ \ w_{j,i}\leq\ \ $} (28);
      \draw[main_node]  (24) edge node{$\ \ \ \ \ \ \ \ \ \ \ \ \ \ \ \ \ w_{i-1,i}\ \ \cdots$} (27);  
       \draw[main_node]  (11) edge node{} (14);  
       \draw[main_node]  (15) edge node{} (18); 
       \draw[main_node]  (19) edge node{$\begin{array}{ccc}< &{}&{}\\{}&\ddots &{}\\{}&{}&< \end{array}\ \ \ \ \ \ \ $} (21);  
       \draw[main_node]  (8) edge node{$\begin{array}{ccc}< &{}&{}\\{}&\ddots &{}\\{}&{}&< \end{array}\ \ \ \ \ \ \ $} (12);  

\draw[main_node,color=blue,dashed,very thick]  (20) edge node{$\ \ \ \ \ \ \ \begin{array}{ccc}{} &{}&>\\{}&\text{\reflectbox{$\ddots$}}&{}\\ >&{}&{} \end{array}$} (23);  
\draw[main_node,color=blue,dashed,very thick]  (9) edge node{$\ \ \ \ \ \ \ \begin{array}{ccc}{} &{}&>\\{}&\text{\reflectbox{$\ddots$}}&{}\\ >&{}&{} \end{array}$} (13);  

       \draw[main_node]  (11) edge node{} (29);  
 \draw[main_node,color=blue,dashed, very thick]  (8) edge node{} (30);  
\draw[main_node,color=blue,dashed,very thick]  (2) edge node{} (31);   

\draw[main_node,color=blue,dashed,very thick]  (17) edge node{} (18);

   \draw[main_node]  (14) edge node{} (33);   
\draw[main_node,color=blue,dashed,very thick]  (23) edge node{} (33);   
   \draw[main_node]  (10) edge node{} (34);   
   \draw[main_node]  (3) edge node{} (35);   

\end{tikzpicture}
\caption{Proof of \Cref{PN}, part (3),(ii). The parallelogram inequality fails: $w_{j,i}\leq w_{j-1,i}$ (edge in red). Only three of four uppermost paths (dashed blue paths) determine the $\{j-1,j\}\times\{i-1,i\}$ minor. The 
 $j,i$ entry of this minor is given by $\UM_{j,i-1} w_{j-1,i}$.}
\label{n'''}
\end{figure}
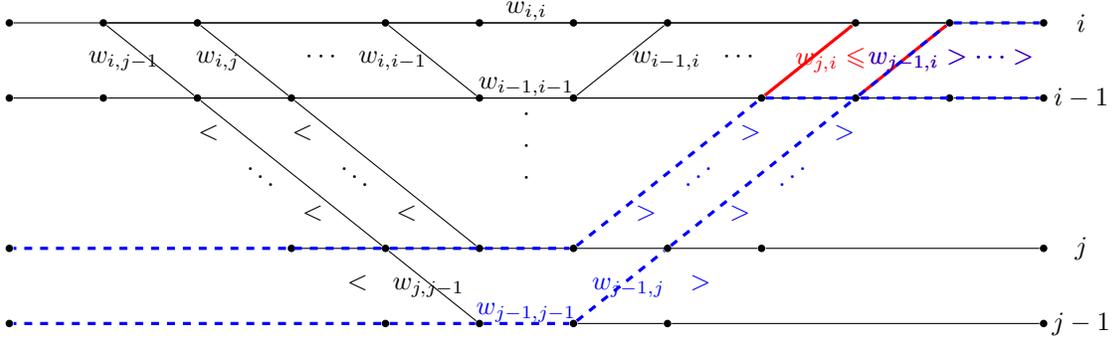
We showed that $W$ satisfies  the strict  parallelogram and trapeze  inequalities. 
Then the uniqueness of $W$ follows from the final part of statement (2) of the theorem.

\end{proof}

As an immediate consequence of the proof of \Cref{PN}, we get
the following result.

\begin{pro}\label{parametrization}

The transformation $\psi$ sends bijectively the set $\mathcal{W}\subset \R^{n\times n}$ of matrices of weights
  that satisfy the weak trapeze and parallelogram inequalities to
  $\TN^\trop(\R)$. Similarly, $\psi$ sends bijectively the interior of $\mathcal{W}$ to $\TP^\trop$.\hfill\qed 
\end{pro}

\begin{rem}
  The planar network in \Cref{ex-easy} illustrates
  the uniqueness result in \Cref{PN}. If $a_{22}>6$,
  then, the matrix $A$ is tropical totally positive,
  and the only matrix of weight which yields $A$
  as a transfer matrix is obtained for $\alpha=a_{2,2}$.
  If $a_{2,2}=6$, then, $A\in\TN^\trop(\R)$, and
  any weight $\alpha\leq 6$ yields the transfer matrix $A$.
\end{rem}
\begin{exa} Let~$G_3$ be the canonical totally connected planar network
  shown in \Cref{fig-g3},  and let~$A$ be its tropical transfer matrix:
$$\left(\begin{array}{ccc}
w_{1,1}     &     w_{1,1}\odot w_{1,2}     &     w_{1,1}\odot w_{1,2}\odot w_{1,3}\\
&&\\
w_{2,1}\odot w_{1,1}     &     w_{2,1}\odot w_{1,1}\odot w_{1,2}\oplus w_{2,2}     &     w_{2,1}\odot w_{1,1}\odot w_{1,2}\odot w_{1,3}\oplus \\
&&w_{2,2}\odot w_{1,3}\oplus w_{2,2}\odot w_{2,3}\\
&&\\
w_{3,1}\odot w_{2,1}\odot w_{1,1}     &     w_{3,1}\odot w_{2,1}\odot w_{1,1}\odot w_{1,2}\oplus    &     w_{3,1}\odot w_{2,1}\odot w_{1,1}\odot w_{1,2}\odot w_{1,3}\oplus\\
                        &    w_{3,1}\odot w_{2,2}\oplus w_{3,2}\odot w_{2,2}          & (w_{3,1}\oplus  w_{3,2})\odot w_{2,2}\odot (w_{2,3}\oplus  w_{1,3})\\
&&\oplus w_{3,3}
\end{array}\right).$$
  \begin{figure}[htbp]\begin{tikzpicture}[main_node/.style={circle,
fill=black,minimum size=0.05em,inner sep=1pt]}]
    \node[main_node] (1) at (0,0) {};
    \node[main_node] (2) at (-1,1)  {};
    \node[main_node] (3) at (2,1) {};
    \node[main_node] (4) at (1,0) {};
    \node[main_node] (5) at (0,1) {};
    \node[main_node] (6) at (1,1) {};
    \node[main_node] (7) at (-2,2) {};
    \node[main_node] (8) at (-1,2) {};
    \node[main_node] (9) at (2,2) {};
    \node[main_node] (10) at (3,2) {};
    \node[main_node] (11) at (0,2) {};
    \node[main_node] (12) at (1,2) {};

    \node[main_node] (13) at (-2,0) {};
    \node[main_node] (14) at (3,0) {};
    \node[main_node] (15) at (-2,1) {};
    \node[main_node] (16) at (3,1) {};

 \draw[main_node]  (1) edge node{$\begin{array}{c}w_{1,1}\\ {}\end{array}$} (4);
 \draw[main_node]  (1) edge node{$w_{2,1}\ \ \ \ \ \ $} (2);   
 \draw[main_node]  (3) edge node{$\ \ \ \ \ \ \ w_{1,2}$} (4);    
\draw[main_node]  (5) edge node{$\begin{array}{c}w_{2,2}\\ {}\end{array}$} (6);  
\draw[main_node]  (2) edge node{} (5);  
\draw[main_node]  (3) edge node{} (6);  
 \draw[main_node]  (2) edge node{$w_{3,1}\ \ \ \ \ \ $} (7);
 \draw[main_node]  (5) edge node{$w_{3,2}\ \ \ \ \ \ $} (8);
 \draw[main_node]  (6) edge node{$\ \ \ \ \ \ \  w_{2,3}$} (9);
 \draw[main_node]  (3) edge node{$\ \ \ \ \ \ \ w_{1,3}$} (10);
\draw[main_node]  (7) edge node{} (10);  

\draw[main_node]  (13) edge node{} (1);
  \draw[main_node]  (14) edge node{} (4);
  \draw[main_node]  (15) edge node{} (2);
  \draw[main_node]  (16) edge node{} (3);  

\draw[main_node]  (8) edge node{$\begin{array}{c}w_{3,3}\\ {}\end{array}$} (9);  
  \end{tikzpicture}
  \caption{The canonical totally connected planar network $G_3$}\label{fig-g3}
\end{figure}
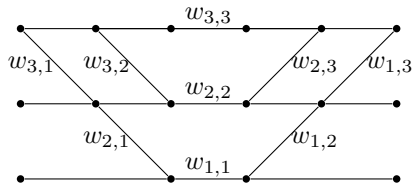

If the weights satisfy the trapeze and parallelogram inequalities,
then, the transfer matrix becomes
\begin{equation}\label{WM}A=\left(\begin{array}{ccccc}
w_{1,1}     &&     w_{1,1}\odot w_{1,2}     &&     w_{1,1}\odot w_{1,2}\odot w_{1,3}\\
&&&&\\
w_{2,1}\odot w_{1,1}     &&     w_{2,2}     &&   w_{2,2}\odot w_{2,3}\\
&&&&\\
w_{3,1}\odot w_{2,1}\odot  w_{1,1}     &&     w_{3,2}\odot w_{2,2}   &&      w_{3,3}
\end{array}\right)\in\TP^\trop\enspace.\end{equation}
\end{exa}

\section{Uniqueness of the factorization of tropical totally positive matrices in terms of Jacobi matrices}\label{S3}
We now apply \Cref{PN} to deduce a uniqueness result
for the factorization of tropical totally positive
matrices.
As recalled above, a classical theorem of Loewner and Whitney
shows that every invertible totally nonnegative matrix is a product
of elementary Jacobi matrices with nonnegative entries,
see~\cite{Loewner},~\cite{Whitney}, and Theorem~12 of~\cite{F&Z}. 
Moreover, if one considers a specific sequence of $n^2$
elementary Jacobi matrices, then one gets a bijective
parametrization. To explain the latter point,
and to derive a tropical analogue of this parametrization
result, some further definitions
are in order. We follow the notation of~\cite{F&Z}. If $s$ is a scalar parameter, and $i\in[n]$, the Jacobi
matrix $x_i(s)$ is the $n\times n$ matrix
\[ x_{i}(s)= I + s E_{i,i+1}
\]
where
$I$ is the identity matrix, and $E_{ij}$ denotes
the $(i,j)$ element of the canonical basis of $\R^{n\times n}$,
with~$1$ in position $(i,j)$ and $0$ elsewhere.
We also set
\[
x_{\barred{i}}(s)= I+sE_{i+1,i}
\]
and
\[x_{\!\circled{\,\,i}\,}(s)= I+(s-1)E_{i,i}
\enspace .
\]
For instance, when $n=3$,
\[
x_{1}(s) =\left(\begin{array}{ccc} 1 & s & 0 \\ 0 & 1 & 0\\ 0 & 0 & 1
\end{array}  \right),
\;
x_{\barred{2}}(s) =\left(\begin{array}{ccc} 1 & 0 & 0 \\ 0 & 1 & 0\\ 0 & s & 1
\end{array}  \right),
\;
x_{\!\circled{\,\,3}\,}(s) =\left(\begin{array}{ccc} 1 & 0 & 0 \\ 0 & 1 & 0\\ 0 & 0 & s
\end{array}  \right).
\]
We consider the alphabet $\mathsf{A}:=\{1, \dots, n-1, 
\circled{\,1}\,, \dots, \circled{\,n}\,, 
\overline{1}, \dots, \overline{n-1}\}$.
Given a word $\bi=(i_1,\dots,i_k)\in \mathsf{A}^k$,
and a vector of parameters $\bs=(s_1,\dots,s_k)$, we define
\[
x_{\bi}(\bs) = x_{i_1}(s_1)\dots x_{i_k}(s_k)\enspace .
\]
The transfer matrix of the planar network $G_n$ coincides
  with $x_{\bi}(\bs)$ for some special word $\bi$,
 where the entries of $\bs$ are precisely the weights $w_{ij}$, listed in a suitable order.

 Indeed, consider the words
\[\bbi^1:=(\overline{n-1)},\;
 \bbi^2:=(\overline{n-2},\overline{n-1}),\; \dots\;,
 \bbi^{n-1}=
 (\bar{1},\dots , \overline{n-1}),
 \]
\[  \bi^{n-1}=
  ({n-1},\dots , {1}),\;
  \dots\;,
  \bi^2:=({n-1},{n-2}),\;
  \bi^1:=({n-1)},
  \]
  and consider the concatenated word, which is of length $n^2$,
  \begin{align}\label{e-def-word}
     \sbi:=\bbi^1\dots\bbi^{n-1}
  \circled{\,1}\dots\circled{\,n}
  \bbi^{n-1}\dots\bbi^{1} \enspace .
  \end{align}
  Observe that these words are constructed in a transparent way
  from the canonical planar network in \Cref{n}. 
  We divide the planar network $G_n$ into three parts, reading the weights in columns. The first part (left of the network) consists of $n-1$ layers with only horizontal or descending arcs, having weights $w_{i,j}$ s.t. $i>j$.
The third part (right of the network) is symmetric to the first one, it consists of $n-1$ layers with only horizontal or ascending arcs, having weights $w_{i,j}$ with $i<j$.

The word $\bbi^1$
  is associated to the first layer, $\bbi^2$ to the second layer, etc.
  Each letter of $\bbi^k$ corresponds to a descending arc in the layer
  $k$, i.e., the only letter $\overline{n-1}$ in $\bbi^1$
  corresponds to the arc from row $n$ to row $n-1$ with weight $w_{n,1}$
  in layer $1$, the two letters $\overline{n-2},\overline{n-1}$ in $\bbi^2$
  correspond to the two descending arcs from rows $n-1$ to $n-2$, and from rows
  $n$ to $n-1$, with respective weights $w_{n-1,1}$ and $w_{n,2}$, etc.
  We now order the parameters $w_{ij}$ layer by layer, consistently with the order   of letters inside each layer. Hence, we set
  \begin{align}\label{W}S_\sbi(W)&:=(w_{n,1},w_{n-1,1},w_{n,2},\dots,
  w_{2,1}, w_{3,1},\dots,w_{n,n-1},
  w_{1,1},\dots, w_{n,n},\\
\nonumber & \qquad  \qquad\qquad w_{n-1,n},\dots,w_{1,2},
  \dots, w_{2,n},w_{1,n-1},w_{1,n})
  \end{align}

  It follows from the interpretation of matrix multiplication
  in terms of concatenation of paths that
  \[
  x_{\sbi}(S_{\sbi}(W))
  \]
  is precisely the transfer matrix of the graph $G_n$ equipped
  with the weights $W=(w_{ij})$.
  Theorem~13 of \cite{F&Z} shows that the map $\bs \to x_{\sbi}(\bs)$ is
  a bijection from $\R_{>0}^{n^2}$ to the set of $n\times n$ totally
  positive matrices. That is, each totally
  positive matrix can be written as $x_{\sbi}(\bs)$ for some
  $\bs\in\R_{>0}^{n^2}$, where  such~$\bs$ is unique.

In other words, the  sequence of Jacobi elementary matrices in the factorization of the transfer matrix,  
corresponds to  the  sequence of weights  in~\eqref{W}, 
in a way that the Jacobi elementary matrix that corresponds to  $w_{i,j}$ is:
\begin{equation}\label{law}x_{\overline{i-1}}(w_{i,j})\text{ where }i>j,\ x_{i}(w_{i,j})\text{ where } i<j\text{ and }x_{\circled{\ i}}(w_{i,i})\text{ where } i=j\text{ (see~\cref{classexa})}.\end{equation}

\begin{exa}\label{classexa} The 
 transfer matrix
 $$ \left(\begin{array}{ccc}
w_{1,1}     &     w_{1,1} w_{1,2}     &     w_{1,1} w_{1,2} w_{1,3}\\
&&\\
w_{2,1} w_{1,1}     &     w_{2,1} w_{1,1} w_{1,2}+ w_{2,2}     &     w_{2,1} w_{1,1} w_{1,2} w_{1,3}+ \\
&&w_{2,2} w_{1,3}+ w_{2,2} w_{2,3}\\
&&\\
w_{3,1} w_{2,1} w_{1,1}     &     w_{3,1} w_{2,1} w_{1,1} w_{1,2}+    &     w_{3,1} w_{2,1} w_{1,1} w_{1,2} w_{1,3}+\\
                        &    w_{3,1} w_{2,2}+ w_{3,2} w_{2,2}          & (w_{3,1}+  w_{3,2}) w_{2,2} (w_{2,3}+  w_{1,3})\\
&&+ w_{3,3}
\end{array}\right)$$
 of the canonical planar network in~\Cref{fig-g3}
 can be factored as 
$$x_{\overline{2}}(w_{3,1}) x_{\overline{1}}(w_{2,1}) x_{\overline{2}}(w_{3,2})
x_{\circled{\ 1}}(w_{1,1}) x_{\circled{\ 2}}(w_{2,2}) x_{\circled{\ 3}}(w_{3,3})
x_{{2}}(w_{2,3}) x_{{1}}(w_{1,2}) x_{{2}}(w_{1,3}),$$
corresponding to the weight sequence $(w_{3,1},w_{2,1},w_{3,2},w_{1,1},w_{2,2},w_{3,3},w_{2,3},w_{1,2},w_{1,3})$, and to the word
$\sbi=\bar 2 \bar 1 \bar 2 \circled{\,1}
\circled{\,2}
\circled{\,3} 2 1 2 $.
\end{exa}
  We now tropicalize the previous constructions.
  We define the matrices $x^\trop_i(s):= I^\trop \oplus s\odot E^\trop_{i,i+1}$,
  where $I^{\trop}$ denotes the tropical identity matrix,
  and $E^{\trop}_{i,j}$ the $(i,j)$ matrix of the canonical tropical
  basis, hence $(x^\trop_i(s))_{ij}= 0$ for $i=j$,
  $(x^\trop_i(s))_{i,i+1}= s$, and $x_{i,j}=-\infty$ otherwise.
  Similarly, $x^\trop_{\bar i}(s):= I^\trop \oplus s\odot E^\trop_{i+1,i}$.
  Finally, $x^\trop_{\!\circled{\,\,i}\,}(s)$ is defined as
  the matrix whose diagonal entries
  are $0,\dots,0,s,0,\dots,0$ ($s$ is in $i$th position), and whose
  off-diagonal entries are $-\infty$, so $x^\trop_{\!\circled{\,\,i}\,}(s)$
  is a tropical diagonal matrix.
  For instance, when $n=3$,
\[
x^\trop_{1}(s)=\left(\begin{array}{ccc} 0 & s & \!-\infty\! \\ \!-\infty\! & 0 & \!-\infty\!\\ \!-\infty\! & \!-\infty\! & 0
\end{array}  \right),
\,
x^\trop_{\barred{2}}(s) =\left(\begin{array}{ccc} 0 & \!-\infty\! & \!-\infty\! \\ \!-\infty\! & 0 & \!-\infty\!\\ \!-\infty\! & s & 0
\end{array}  \right),
\,
x^\trop_{\!\circled{\,\,3}\,}(s) =\left(\begin{array}{ccc} 0 & \!-\infty\! & \!-\infty\! \\ \!-\infty\! & 0 & \!-\infty\!\\ \!-\infty\! & \!-\infty\! & s
\end{array}  \right).
\]
Then, for
  $\bi=(i_1,\dots,i_k)\in \mathsf{A}^k$, given
a vector of real parameters $\bs=(s_1,\dots,s_k)$, we define
\[
x^\trop_{\bi}(\bs) := x^\trop_{i_1}(s_1)\odot \dots \odot x^\trop_{i_k}(s_k)
\]

\begin{cor}\label{cor-unique}
  If $A$ is a tropical totally positive matrix, there exists
  a unique vector $\bs\in \R^{n^2}$ such that $x_{\sbi}^\trop(\bs)=A$,
  where $\sbi$ is the canonical word~\eqref{e-def-word}.
\end{cor}
\begin{proof}
  By construction of the word $\sbi$, $x^\trop(\bs(W))$ is precisely
  the tropical transfer matrix of the planar network $G_n$ equipped
  with tropical weights $w_{ij}$. Then, the existence and uniqueness
  of the weights $w_{ij}$ follows from \Cref{PN}.
\end{proof}
\begin{exa}\label{tropexa}In analogy with~\cref{classexa}, if the weights in~\Cref{fig-g3} satisfy the strict trapeze and parallelogram inequalities, then 
the factorization of its $\TP^\trop$
 transfer matrix in~\eqref{WM} is $$x_{\overline{2}}(w_{3,1})\odot x_{\overline{1}}(w_{2,1})\odot x_{\overline{2}}(w_{3,2})\odot
x_{\circled{\ 1}}(w_{1,1})\odot$$$$ \ \ x_{\circled{\ 2}}(w_{2,2})\odot x_{\circled{\ 3}}(w_{3,3})\odot
x_{{2}}(w_{2,3})\odot x_{{1}}(w_{1,2})\odot x_{{2}}(w_{1,3}),$$
and the weights $w_{ij}$ are uniquely determined by the transfer matrix.\end{exa}

We next translate \Cref{cor-unique} in terms
of matrices over the nonarchimedean field $\mathbb{K}$.
\begin{cor}\label{cor-unique2}
  Suppose that $\mathcal{A}\in\mathbb{K}^{n\times n}$ 
  is the transfer matrix of the planar network $G_n$,
  equipped with weights $w_{ij}$ in the nonarchimedean field $\mathbb{K}$.
  Assume in addition that $\val(\mathcal{A})$ is
  a tropical totally positive matrix. Then, the valuation
  of the weights $w_{ij}$ can be computed knowing only
  the valuations of the entries $\mathcal{A}_{ij}$.
\end{cor}
\begin{proof}
  By assumption, we have $\mathcal{A}=x_{\sbi}(S_\sbi(W))$, where $W=(w_{ij})\in \mathbb{K}^{n\times n}$ is the collection of weights. Since the valuation
  is a morphism from $\mathbb{K}_{\geq 0}$ to $\rmax$, we deduce
  that $\val(\mathcal{A})=x^\trop_{\sbi}(S_\sbi(\val W))$. By \Cref{cor-unique},
  the sequence of tropical weights $S_\sbi(\val W)$, or equivalently,
  the matrix $\val W$, is uniquely
  determined by the matrix $\val(\mathcal{W})$.

\end{proof}
\begin{rem}
  The assumption that $\val(A)\in \TP^\trop$ is essential
  in \Cref{cor-unique2}. Consider, for $a>0$,
  \[ \TP(\K)\ni  \mathcal{A}:= \left(\begin{array}{cc}
    1 & 1\\
    1 & 1 + t^{-a}
  \end{array}\right)
  =
  \left(\begin{array}{cc}
    1 & 0\\
    1 & 1
  \end{array}\right)
  \left(\begin{array}{cc}
    1 & 0 \\ 0 & t^{-a}
  \end{array}\right)
  \left(\begin{array}{cc}
    1 & 1\\
    0 & 1
  \end{array}\right)
  = x_{\sbi}(1,1,t^{-a},1)\enspace .
  \]
  So the valuation $-a= \val w_{22}$ in the bottom-right entry of the central factor cannot be recovered from $\val \mathcal{A}= \left(\begin{smallmatrix}0 & 0 \\ 0 & 0 \end{smallmatrix}\right)$.
  In contrast, the weight $w_{22}\in \K$ is uniquely determined
  by $\mathcal{A}$. This loss of uniqueness modulo valuation
  is easily explained: $w_{22}$ can be expressed from the entries
  $\mathcal{A}_{ij}$ by an expression involving a subtraction,
  and this expression does not ``tropicalize''. 
\end{rem}
Fomin and Zelevinski have characterized in~\cite{fominbruhat,F&Z}
the words $\bj$ in the alphabet $\mathsf{A}$ that
yield (bijective) parametrizations of the set of totally positive
matrices by the standard positive cone. These words are called
{\em factorization schemes}. We recall briefly one of their characterizations.
This characterization
involves the presentation of the symmetric group in $n$ letters by
the transpositions $\tau_1,\dots,\tau_{n-1}$, permuting consecutive indices.
These transpositions satisfy $\tau_i^2=1$, $\tau_i\tau_j=\tau_j\tau_i$
for $|i-j|\geq 2$, and $\tau_j\tau_i\tau_j=\tau_i\tau_j\tau_i$. A word
in the letters $\tau_i$ is said to be {\em reduced} if it is of minimal length
among all the words that yield a given permutation. This minimal
length is precisely the number of inversions of the permutation.
Among all permutations, there is one permutation $\sigma$ maximizing the number
of inversions: $\sigma$ reverses the order of $1,\dots,n$.
A word $\bj$ is a {\em factorization scheme} if it is a shuffle of a reduced word
in the alphabet $\{1,\dots,n\}$ representing
the permutation $\sigma$, of another reduced word
in the alphabet $\{\bar 1,\dots,\bar n\}$ representing
the same permutation, and of any word obtained
from $\circled{\,1}\,, \dots, \circled{\,n}$ by a reordering.
Recall that a shuffle of two words $\bj$ and $\bk$ is an interleaving of these words (it contains all the characters of $\bj$ and $\bk$, counting with multiplicities, and respects the order of all characters in individual words).
The word $\sbi$ constructed in~\eqref{e-def-word}
is an example of factorization scheme. 
It is shown
in Theorem~23 of~\cite{F&Z} that
if $\bj$ is a factorization scheme,
then the map $\bs\mapsto x_{\bj}(\bs)$ is a bijection
from $\R_{>0}^{n\times n}$ to $\TP$. In this way, different
factorization schemes lead to different parametrizations
of totally positive matrices. 
If $\bj$ and $\bk$ are factorization schemes, then
one can consider the transition map $R_{\bj,\bk}$
defined by
\begin{align}
  \label{e-def-R}
  x_{\bj}(\bs) = x_{\bk} (R_{\bk,\bj}(\bs)) \enspace ,
  \forall \bs\in \R_{>0}^{n^2} \enspace ,
  \end{align}
i.e., the map which associates to $\bs\in \R_{>0}^{n^2}$ the unique
$\bs'\in \R_{>0}^{n^2}$ such that $x_{\bk}(\bs')=x_{\bj}(\bs)$.
The map $R_{\bj,\bk}$
is a rational transformation that is {\em subtraction-free}, meaning
that every coordinate map of  $R_{\bj,\bk}$
can be expressed in terms of the parameters
$s_i$ using only sums, products and divisions.
Moreover,
\begin{align}
  R_{\bj,\bk}\circ R_{\bk,\bj} = I_{\R^{n^2}_{>0}} \enspace  .
  \label{e-inverse}
\end{align}
The same is true in the tropical setting.
\begin{thm}[Corollary of~\cite{F&Z}, see also~\cite{berensteinparametrization}]\label{coroFZ}
  For all factorization schemes $\bj$ and $\bk$, there is a tropical
  rational transformation $R^\trop_{\bk,\bj}: \R^{n^2}\to \R^{n^2}$ such that
  \begin{align}
    x_{\bj}^\trop(\bs) = x_{\bk}^\trop (R^\trop_{\bk,\bj}(\bs)) \enspace .
    \label{e-change}
  \end{align}
  Moreover,
  \begin{align}
    R^\trop_{\bk,\bj}\circ R^\trop_{\bj,\bk}=I_{\R^{n^2}}
    \enspace .\label{e-identity2}
    \end{align}
\end{thm}
\begin{proof}
This is indeed an immediate consequence of the proof
of Theorem~23 of ~\cite{F&Z}, using
the observation, already made in~\cite{berensteinparametrization},
that a subtraction free rational fraction over the
set of positive numbers defines a tropical rational
fraction over the reals. Moreover, if a subtraction free
rational transformation $R$ coincides with the identity
map on $\R_{>0}^{n^2}$, then its numerator and denominator seen
as formal polynomials must coincide, which entails
that the corresponding tropical rational transformation
is the identity on $\R^{n^2}$. Hence,~\eqref{e-identity2} follows
from~\eqref{e-inverse}.
\end{proof}

\begin{exa}
 To give an example of the correspondence
 between $R_{\bj,\bk}$ and $R_{\bj,\bk}^{\trop}$,
 used in the proof of \Cref{coroFZ},
  consider the following
  commutation relation from~\cite{F&Z}, which holds for $i\in[n-1]$ and $j=i+1$,
\begin{align}
x_i (s_1)\, x_{\circi}(s_2)\, x_{\circj}(s_3)\, x_{\overline i}
(s_4)
= x_{\overline i}(s'_1) \, x_{\circi}(s'_2)\,
x_{\circj}(s'_3) \, x_i(s'_4) \,, \label{e-classical}
\end{align}
where 
\[
s'_1\!=\!\frac{s_3s_4}{T}\,,\ 
s'_2\!=\!{T}\,,\ 
s'_3\!=\!\frac{s_2s_3}{T}\,,\ 
s'_4\!=\!\frac{s_1s_3}{T}\,,\ 
T\!=\!s_2+s_1s_3s_4\enspace .
\]
This defines a map $R: s\mapsto s'$. 
The tropical analogue of relation~\eqref{e-classical} is
\[
x^\trop_i (s_1)\odot
x^\trop_{\circi}(s_2)\odot
x^\trop_{\circj}(s_3)\odot
x^\trop_{\overline i} (s_4)
 = 
x^\trop_{\overline i}(s'_1) \odot
x^\trop_{\circi}(s'_2)\odot
x^\trop_{\circj}(s'_3)  \odot
x^\trop_i(s'_4) , 
\]
where 
\[
s'_1={s_3+s_4}-{T}\,,\ 
s'_2={T}\,,\ 
s'_3={s_2+s_3}-{T}\,,\ 
s'_4={s_1+s_3}-{T}\,,\ 
T=\max(s_2,s_1+s_3+s_4)\,.
\]
\end{exa}

\begin{cor}\label{cor-uniquenew}
  If $A$ is a tropical totally positive matrix, and if $\bj$
  is an arbitrary factorization scheme, there exists
  a unique vector $\bs\in \R^{n^2}$ such that $x_{\bj}^\trop(\bs)=A$.
\end{cor}
\begin{proof}
  By \Cref{coroFZ}, there exists an invertible tropical rational
  transformation $R: \R^{n^2}\to \R^{n^2}$ such that $x_{\bj}(s)=x_{\bi}(R(s))$.
  Then, the corollary follows from \Cref{cor-unique}.
  \end{proof}

\bibliographystyle{alpha}

\bibliography{tropical}

\def\cprime{$'$} \def\cprime{$'$}
\begin{thebibliography}{BCOQ92}

\bibitem[ABG07]{MPA}
M.~Akian, R.~Bapat, and S.~Gaubert.
\newblock Max-plus algebras.
\newblock In Leslie Hogben, editor, {\em Handbook of linear algebra}, Discrete
  Mathematics and its Applications (Boca Raton). CRC Press, Boca Raton, FL,
  second edition, 2007.
\newblock Chapter 25.

\bibitem[ABGJ15]{benchimol2013}
X.~Allamigeon, P.~Benchimol, S.~Gaubert, and M.~Joswig.
\newblock Tropicalizing the simplex algorithm.
\newblock {\em SIAM J. Disc. Math.}, 29(2):751--795, 2015.

\bibitem[AGG09]{LDTS}
M.~Akian, S.~Gaubert, and A.~Guterman.
\newblock Linear independence over tropical semirings and beyond.
\newblock In {\em Tropical and idempotent mathematics}, volume 495 of {\em
  Contemp. Math.}, pages 1--38. Amer. Math. Soc., Providence, RI, 2009.

\bibitem[AGG14]{AGG14}
M.~Akian, S.~Gaubert, and A.~Guterman.
\newblock Tropical {C}ramer determinants revisited.
\newblock In {\em Tropical and idempotent mathematics and applications}, volume
  616 of {\em Contemp. Math.}, pages 1--45. Amer. Math. Soc., Providence, RI,
  2014.

\bibitem[AGS16]{skomra}
Xavier Allamigeon, St\'ephane Gaubert, and Mateusz Skomra.
\newblock Tropical spectrahedra, 2016.
\newblock arXiv:1610.06746.

\bibitem[AGW09]{agw04b}
M.~Akian, S.~Gaubert, and C.~Walsh.
\newblock The max-plus {M}artin boundary.
\newblock {\em Doc. Math.}, 14:195--240, 2009.

\bibitem[Ale13]{alessandrini2013}
D.~Alessandrini.
\newblock Logarithmic limit sets of real semi-algebraic sets.
\newblock {\em Adv. Geom}, 13:155--190, 2013.

\bibitem[And87]{TPM}
T.~Ando.
\newblock Totally positive matrices.
\newblock {\em Linear Algebra Appl.}, 90:165--219, 1987.

\bibitem[BCOQ92]{BCOQ92}
F.~Baccelli, G.~Cohen, G.J. Olsder, and J.P. Quadrat.
\newblock {\em Synchronization and linearity}.
\newblock Wiley Series in Probability and Mathematical Statistics: Probability
  and Mathematical Statistics. John Wiley \& Sons, Ltd., Chichester, 1992.
\newblock An algebra for discrete event systems.

\bibitem[BFZ96]{berensteinparametrization}
A.~Berenstein, S.~Fomin, and A.~Zelevinsky.
\newblock Parametrizations of canonical bases and totally positive matrices.
\newblock {\em Adv. Math.}, 122(1):49--149, 1996.

\bibitem[BKR96]{BKR}
R.~E. Burkard, B.~Klinz, and R.~Rudolf.
\newblock Perspectives of {M}onge properties in optimization.
\newblock {\em Discrete Appl. Math.}, 70(2):95--–161, 1996.

\bibitem[BS95]{shader}
R.~A. Brualdi and B.~L. Shader.
\newblock {\em Matrices of Sign-Solvable Linear Systems}.
\newblock Number 116 in Cambridge Tracts in Mathematics. Cambridge University
  Press, 1995.

\bibitem[BSS07]{BSS}
P.~Butkovi{\v{c}}, H.~Schneider, and S.~Sergeev.
\newblock Generators, extremals and bases of max cones.
\newblock {\em Linear Algebra Appl.}, 421(2-3):394--406, 2007.

\bibitem[But10]{butkovicbook}
P.~Butkovi{\v{c}}.
\newblock {\em Max-linear systems: theory and algorithms}.
\newblock Springer Monographs in Mathematics. Springer-Verlag London, Ltd.,
  London, 2010.

\bibitem[DKK09]{danilov}
V.I. Danilov, A.V. Karzanov, and G.A. Koshevoy.
\newblock Tropical pl\"ucker functions and their bases.
\newblock In G.~L. Litvinov and S.~N. Sergeev, editors, {\em Tropical and
  idempotent mathematics}, volume 495 of {\em Contemporary Mathematics}, pages
  x+382. American Mathematical Society, Providence, RI, 2009.

\bibitem[DS04]{DS}
M.~Develin and B.~Sturmfels.
\newblock Tropical convexity.
\newblock {\em Doc. Math.}, 9:1--27, 2004.

\bibitem[DSS05]{DSS}
M.~Develin, F.~Santos, and B.~Sturmfels.
\newblock On the rank of a tropical matrix.
\newblock In {\em Combinatorial and computational geometry}, volume~52 of {\em
  Math. Sci. Res. Inst. Publ.}, pages 213--242. Cambridge Univ. Press,
  Cambridge, 2005.

\bibitem[DY07]{develin2007}
Mike Develin and Josephine Yu.
\newblock Tropical polytopes and cellular resolutions.
\newblock {\em Experiment. Math.}, 16(3):277--292, 2007.

\bibitem[Fie06]{FIEDLER}
M.~Fiedler.
\newblock Subtotally positive and {M}onge matrices.
\newblock {\em Linear Algebra Appl.}, 413:177--188, 2006.

\bibitem[FJ11]{Fallat&Johnson}
S.~M. Fallat and C.~R. Johnson.
\newblock {\em Totally Nonnegative Matrices}.
\newblock Princeton Series in Applied Mathematics. Princeton university press,
  2011.

\bibitem[FZ99]{fominbruhat}
Sergey Fomin and Andrei Zelevinsky.
\newblock Double bruhat cells and total positivity.
\newblock {\em Journal of the American Mathematical Society}, 12(2):335--380,
  1999.

\bibitem[FZ00]{F&Z}
S.~Fomin and A.~Zelevinsky.
\newblock Total positivity: Tests and parametrizations.
\newblock {\em The Mathematical Intelligencer}, 22(1):23--33, 2000.

\bibitem[Gau92]{G.Ths}
S.~Gaubert.
\newblock {\em Th\'eorie des syst\`emes lin\'eaires dans les dio\"\i des}.
\newblock Phd dissertation, \'Ecole des Mines de Paris, Paris, July 1992.

\bibitem[GB99]{NSM}
S.~Gaubert and P.~Butkovi\v{c}.
\newblock Sign-nonsingular matrices and matrices with unbalanced determinant in
  symmetrised semirings.
\newblock {\em Linear Algebra Appl.}, 301(1-3):195--201, 1999.

\bibitem[GK35]{Gantmacher&Krein}
F.~R. Gantmacher and M.~G. Krein.
\newblock Sur les matrices oscillatoires.
\newblock {\em C. R. Acad. Sci. (Paris)}, 201:577--579, 1935.

\bibitem[GK07]{GK}
S.~Gaubert and R.~Katz.
\newblock The {M}inkowski theorem for max-plus convex sets.
\newblock {\em Linear Algebra and Appl.}, 421:356--369, 2007.

\bibitem[GM96]{kluwertotalpositivity}
M.~Gasca and C.~A. Micchelli, editors.
\newblock {\em Total positivity and its applications}, volume 359 of {\em
  Mathematics and its Applications}. Kluwer Academic Publishers Group,
  Dordrecht, 1996.

\bibitem[GN18]{GN17}
S.~Gaubert and A.~Niv.
\newblock Tropical totally positive matrices.
\newblock {\em Journal of Algera}, 515:511--544, 2018.

\bibitem[IMS07]{TAG}
I.~Itenberg, G.~Mikhalkin, and E.~Shustin.
\newblock {\em Tropical algebraic geometry}, volume~35 of {\em Oberwolfach
  Seminars}.
\newblock Birkh\"auser Verlag, Basel, 2007.

\bibitem[JSY18]{yuetal}
Ph. Jell, C.~Scheiderer, and J.~Yu.
\newblock Real tropicalization and analytification of semialgebraic sets, 2018.
\newblock arXiv:1810.05132.

\bibitem[KM59]{KMc}
S.~Karlin and G.~McGregor.
\newblock Coincidence probabilities.
\newblock {\em Pacific J.~Math.}, 9:1141--I164, 1959.

\bibitem[Loe55]{Loewner}
C.~Loewner.
\newblock On totally positive matrices.
\newblock {\em Mathematische Zeitschrift}, 63(1):338--340, 1955.

\bibitem[Mar10]{MARK}
T.~Markwig.
\newblock A field of generalised {P}uiseux series for tropical geometry.
\newblock {\em Rend.~Semin.~Mat.~Univ.~Politechnico di Torino}, 68(1):79--92,
  2010.

\bibitem[MS15]{MacStur}
D.~Maclagan and B.~Sturmfels.
\newblock Introduction to tropical geometry.
\newblock {\em AMS Graduate Studies in Mathematics}, 161, 2015.

\bibitem[Plu90]{LS}
M.~Plus.
\newblock Linear systems in $(\max,+)$-algebra.
\newblock In {\em Proceedings of the 29th Conference on Decision and Control},
  Honolulu, Dec. 1990.

\bibitem[Pos06]{POST}
A.~Postnikov.
\newblock Total positivity, {G}rassmannians, and networks, 2006.
\newblock \arxiv{math/0609764}.

\bibitem[SW05]{SW05}
D.~Speyer and L.~Williams.
\newblock The tropical totally positive {G}rassmannian.
\newblock {\em J.~Alg.~Comb.}, 22(2):189--210, 2005.

\bibitem[Whi52]{Whitney}
A.~M. Whitney.
\newblock A reduction theorem for totally positive matrices.
\newblock {\em J. d'Analyse Math.}, 2:88--92, 1952.

\bibitem[Yu15]{YU}
J.~Yu.
\newblock Tropicalizing the positive semidefinite cone.
\newblock {\em Proc.~Amer.~Math.~Soc.}, 143:1891--1895, 2015.

\end{thebibliography}

\end{document}